\documentclass[11pt,reqno]{amsproc}

\usepackage[sc]{mathpazo}\renewcommand{\mathbf}{\mathbold}
\normalfont
\usepackage[T1]{fontenc}

\usepackage{misc/packages}
\newcommand{\C}{\mathbb{C}}
\newcommand{\Z}{\mathbb{Z}}
\newcommand{\N}{\mathbb{N}}
\newcommand{\Q}{\mathbb{Q}}
\newcommand{\R}{\mathbb{R}}

\newcommand{\bA}{\mathbb{A}}
\newcommand{\bB}{\mathbb{B}}
\newcommand{\bU}{\mathbb{U}}

\newcommand{\bM}{\mathbb{M}}

\newcommand{\bI}{\mathbb{I}}

\newcommand{\cC}{\mathcal{C}}

\newcommand{\apl}{\mathcal{A}_{PL}}

\newcommand{\Set}{\mathsf{Set}}

\newcommand{\sSet}{\mathsf{sSet}}
\newcommand{\op}{\mathsf{op}}

\newcommand{\Hom}{\mathsf{Hom}}
\newcommand{\sCh}{\mathsf{Ch}^*_\Q}

\newcommand{\sVec}{\mathsf{Vec}_\Q}
\newcommand{\Ker}{\mathrm{Ker} \,}

\newcommand{\Ch}{\mathsf{Ch}^*}
\newcommand{\CDGA}{\mathsf{CGDA}}
\newcommand{\Mod}{\mathsf{Mod}}
\newcommand{\Di}{\mathbb{D}}
\newcommand{\Fun}{\mathsf{Fun}}
\newcommand{\I}{\mathbb{I}}
\newcommand{\J}{\mathbb{J}}
\newcommand{\Sp}{\mathbb{S}}

\newcommand{\W}{\mathbb{W}}
\newcommand{\X}{\mathbb{X}}
\newcommand{\K}{\mathbb{K}}
\newcommand{\Y}{\mathbb{Y}}
\newcommand{\V}{\mathbb{V}}
\newcommand{\Inj}{\mathsf{Inj}}
\newcommand{\Cof}{\mathsf{Cof}}
\newcommand{\Cell}{\mathsf{Cell}}

\newcommand{\colim}{\mathsf{colim}}

\newcommand{\DDelta}{\mathbf{\Delta}}

\newcommand{\tame}{\mathsf{tame}}
\newcommand{\gr}{\mathsf{gr}}
\newcommand{\Lan}{\mathsf{Lan}}

\newtheorem*{theorem*}{Theorem}
\newtheorem{theorem}{Theorem}[section]

\newtheorem{lemma}[theorem]{Lemma}
\newtheorem{corollary}[theorem]{Corollary}

\newtheorem{proposition}[theorem]{Proposition}
\newtheorem{example}[theorem]{Example}

\theoremstyle{definition}
\newtheorem{definition}[theorem]{Definition}

\theoremstyle{remark}
\newtheorem{remark}[theorem]{Remark}
\newtheorem{notation}[theorem]{Notation}

\numberwithin{equation}{section}

\title{Interval-sphere model structures}
\date{\today}
\author{Kathryn Hess, Samuel Lavenir, Kelly Maggs}

\begin{document}

\maketitle
% \vspace{-32pt}
\begin{center}
\begin{small}
    \today
\end{small}
\end{center}

\maketitle

\begin{abstract}
    The bedrock of persistence theory over a single parameter is decomposition of persistence modules into intervals. In \cite{hess2024celldecompositionspersistentminimal}, the authors leveraged interval decomposition to produce a cell decomposition of the minimal model of a simply connected copersistent space. The key tool was a technique called interval surgery, which involves the gluing of intervals to a persistent CDGA by means of algebraic cell attachments. In this article, we define a compact, combinatorial model categorical structure that contextualizes interval surgery as a genuine model-categorical cell attachment. We show that our new model structure is neither the injective nor the projective one and that cofibrancy is closely linked to the notion of tameness in persistence theory and algebraic notions of compactness. 
\end{abstract}

\tableofcontents

\section{Introduction}

For a field $\mathbf{k}$, functors $\mathbb{V} \colon (\R_{+},\le) \to\mathsf{Vec}_{\mathbf{k}}$ from the poset of non‑negative reals to the category of $\mathbf{k}$-vector spaces are called (real‑indexed) \emph{persistence modules}.  They play a central role in topological data analysis (TDA), as algebraic invariants of the homology of sub‑level set filtrations of compact manifolds \cite{Barannikov_1994,Carlsson05,Oudot_2015} and of filtered Vietoris–Rips complexes of finite metric spaces \cite{Robins_1999,Edelsbrunner_Letscher_Zomorodian}.  The atomic pieces of this theory are the \emph{interval modules}
\[
  \bI_{[s,t)} \colon (\R_{+},\le)\;\longrightarrow\;\mathsf{Vec}_{\mathbf{k}},
  \quad
  \bI_{[s,t)}(r)=
  \begin{cases}
    \mathbf{k},&s\le r< t,\\
    0,&\text{otherwise.}
  \end{cases}
\]
Under mild tameness assumptions, persistence modules decompose into direct sums of interval modules:  \cite{Webb_1985,Carlsson05,chazal2016structure,Crawley-Boevey_2015} as
\[
  \mathbb{V} \;\cong\;\bigoplus_{\gamma}\bI_{[s_{\gamma},\,t_{\gamma})}.\] 

In \cite{hess2024celldecompositionspersistentminimal} we introduced \emph{interval surgery}, a procedure for constructing persistent algebraic objects in arbitrary categories by iteratively attaching interval modules as cells.  In the setting of persistent commutative differential graded algebras (CDGAs), this procedure yields explicit Sullivan minimal models filtered by interval attachments, under simple‑connectedness and tameness hypotheses.  Interval modules thus act as the fundamental “cells’’ of a broad class of persistent algebraic structures.

In this article we provide the model‑categorical underpinnings for interval surgery.  We equip the category of persistent cochain complexes over the rationals with a new combinatorial, compactly generated model structure, the \emph{interval–sphere model structure} (Theorem \ref{modelstructurechaincomplexes}, \Cref{cofib-gen}). This model structure is neither the injective nor the projective one, largely due to the topology of the indexing category $\R_+$ (\Cref{not-projective}, \Cref{non-injective-model-example}). Its cofibrant generators are precisely the interval attachments, meaning that cofibrant objects are closely related to those built from intervals. We show that every monomorphism between tame objects is a cofibration (\Cref{injectivity-compactness}), as a consequence of the close relation between tameness and interval-decomposability.

We then transfer the interval–sphere structure along the standard free–forgetful adjunction to persistent CDGAs (\Cref{modelstructureCDGA}), enabling us to exhibit the the persistent Sullivan minimal model as a genuine cellular presentation (\Cref{persistent-min-moodel}).  Finally, we prove that there exists a Quillen pair between the category of persistent CDGAs (equipped with the interval–sphere structure) and that of copersistent simplicial sets (equipped with the \(\Q\)‑projective structure) in Theorem \ref{persistent-quillen-pair}.

Where much of the TDA literature generalizes the \emph{indexing} category (e.g.,\ to multiparameter persistence \cite{Cerri_Fabio_Ferri_Frosini_Landi_2013,Blumberg_Lesnick_2022,Blumberg_Lesnick_2023}), we instead enrich the \emph{target} category, while fixing as source poset \(\R_+\).  By treating interval modules as cells at the cochain level, we lay a blueprint for transporting both the interval–sphere model structure and the interval surgery process into diverse algebraic and homotopical contexts.

\subsection*{Related work} The idea of focusing on persistent chain/cochain complexes goes back to the origins of persistence theory \cite{Carlsson05}. The use of model categories to describe persistent chain complexes, with a focus on tameness and decomposing cofibrant objects, is more recent \cite{giunti_2021,chacholski2024abelianmodelstructurestame,Chachlski2023}. Model categories have also been used in the context of stability \cite{Blumberg_Lesnick_2023,Lanari_Scoccola_2023}.

\subsection*{Acknowledgements} The authors thank Jérôme Scherer for his helpful, detailed feedback on earlier versions of this manuscript.  K.M. was supported by the European Union’s Horizon 2020 Research and Innovation Program under Marie Skłodowska-Curie Grant Agreement No 859860. S.L. was supported by Swiss National Science Foundation, grant/award number: 200020 18858. 

\begin{comment}
For a finite sequence of finite simplicial complexes
$ S_0 \xrightarrow{f_0} \ldots \xrightarrow{f_{n-1}} S_n,$
applying the rational homotopy group functor $\pi_k( \, \cdot \, ) \otimes \Q$ for any $k \geq 0$ to induces a finite sequences of rational vector spaces
\begin{equation} \label{rational_persistence_module}
\pi_k(S_0) \otimes \Q \xrightarrow{(f_0)_*} \ldots \xrightarrow{(f_{n-1})_*} \pi_k(S_n) \otimes \Q.
\end{equation}
The primary goal of this document is provide a roadmap to compute the persistence diagram of \ref{rational_persistence_module} in the case that each $S_i$ is simply connected.\\

\textbf{Related Work}
\end{comment} 

\subsection{Notation and conventions}

\begin{itemize}
    \item We apply cohomological conventions, i.e., differentials increase degree.
    \item We use the word \textit{space} as a synonym for \textit{simplicial set}.
    \item All vector spaces are over the field $\Q$ of rationals.
    %\item  We write $V^\vee$ for the dual vector space $\Hom_\Q(V, \Q)$.
    \item We denote by $\Ch_\Q$ the category of non-negatively graded cochain complexes over $\Q$ and by $\CDGA_\Q$ the category of commutative differential graded algebras (CDGAs) over $\Q$, i.e., the category of commutative monoids in $\Ch_\Q$.
    \item For $(C,d) \in \Ch_\Q$, the shifted complex $(C[n],d')$ is defined by $C[n]^k=C^{k-n}$ with differentials $d'_k=(-1)^n d_k$.
    \item A CDGA $\mathcal{A}$ is said to be connected if $H^0(\mathcal{A})=\Q$ and simply-connected if $H^1(\mathcal{A}) = 0$ additionally. 
    \item We use the term \textit{finite type} in three settings.
    \begin{itemize}
        \item A simplicial set is finite type if it has finitely many non-degenerate simplices.
        \item A cochain complex is finite type if it is degree-wise finite dimensional.
        \item A CDGA is finite type if its cohomology is degree-wise finite dimensional.
    \end{itemize}
\end{itemize}

\section{Persistent cochain complexes}

In this section, we review the background material on persistent cochain complexes. We highlight notions of tameness introduced in \cite{giunti_2021,chacholski2024abelianmodelstructurestame}. We place emphasis on the  commutative algebra perspective on persistence \cite{Carlsson05,miller2020homologicalalgebramodulesposets}, and particular, notions of algebraic compactness pivotal to understanding the structure of cofibrant objects in later section. K.M. was supported by the European Union’s Horizon 2020 Research and Innovation Program under Marie Skłodowska-Curie Grant Agreement No 859860. S.L. was supported by Swiss National Science Foundation, grant/award number: 200020 18858. 

\subsection{Cochain complexes} We start by setting out notation for usual cochain complexes of vector spaces. 

\subsubsection{Cochain spheres and disks} To every vector space $V$, we can associate sphere $S^k(V)$ and disk $D^k(V)$ cochain complexes for every integer $k\geq 0$. The complex $S^k(V)$ consists of a copy of $V$ in degree $k$, while the complex $D^k(V)$ consists of a copy of $V$ in each of degrees $(k-1)$ and $k$, where the only non zero differential is the identity. When $k=0$, we adopt the convention that $D^0(V)=0$. We also write $S^k=S^k(\Q)$ and $D^k=D^k(\Q)$. Similarly, given a graded $\Q$-vector space $V=\bigoplus_{k\geq 0} V_k$, we can define associated sphere $S(V)$ and disk $D(V)$ complexes. These are cochain complexes  defined by $$
S(V)=\bigoplus_{k\geq 0} S^k(V_k) 
\:\:\:\:\:\:\:,\:\:\:\:\:\:\:
D(V)=\bigoplus_{k> 0} D^k(V_k).
$$
Note that there are evident inclusions $S(V) \subseteq D(V)$ and that $D(V)/S(V)=S(V[-1])=S(V)[-1]$ for any choice of graded vector space $V$. 

\subsubsection{Model structures}
\label{modelstructurechaincomplexes}
There is a well known (projective) model structure on the category $\Ch_\Q$ of cochain complexes, of which the fibrations are the degree-wise surjective maps, and the weak equivalences are the quasi-isomorphisms, i.e., maps inducing isomorphisms on cohomology (cf. \cite{hess05}). The cofibrations can be described as those cochain maps that are injective in degrees $>0$. This model structure is combinatorial \cite[Prop.3.13]{BEKE2000}, i.e. locally presentable and generated by a small set of cofibrations, with the following sets of generating cofibrations and trivial cofibrations.
$$
\begin{aligned}
I &= \{ \: S^k \hookrightarrow D^k \:,\: 0\to S^0 \: | \: k\in \N \:\} \\
J &= \{ \: D^k \to 0 \: | \: k\in \N \:\}.
\end{aligned}
$$

\subsection{Persistence theory and tameness} 

Before generalizing the previous construction from the level of maps to persistent CDGAs, we review the basic notation and terminology about persistent objects in a category. We recall the notion of tame persistent objects from \cite{giunti_2021},  introduce the persistent spheres and disks, and finish with a description of compact objects in several persistent categories of interest.

\begin{notation} Let $\R_+$ be category associated to the poset $[0,\infty).$
   \begin{itemize} 
   \item The category $p\cC$ of \textit{persistent objects} in a category $\cC$ is the functor category $\Fun(\R_+, \cC)$.
   \item If letters $X,Y,Z, \cdots$ are used for the objects of the category $\cC$, we use corresponding boldface symbols $\X,\Y,\Z, \cdots $ to name the objects of $p\cC$.
   \item Dually, we label the functor category $\Fun(\R_-, \cC)$ of \textit{copersistent objects }as $p^*\cC$, where $\R_-=(-\infty, 0]\cong \R_+^\op$ is the dual poset to $\R_+$. 
   \end{itemize}
\end{notation}

\subsection{Tameness}

The use of $\R_+=[0,\infty)$ as the indexing category is central to the definition of the interleaving distance and stability theorems. However, in practice and applications, we take as input a finite sequence of spaces/chain complexes/vector spaces, so the resulting functor changes at only finitely many index values. The following definitions of discretizations and tameness are presented in \cite{giunti_2021}, and formalise this notion. 

\begin{definition}
    Let $\X \in pC$ be a persistent object of $\cC$.
    \begin{enumerate}
        \item A finite sequence $t_0 < t_1 < \ldots < t_n$ of non-negative reals is said to \textit{discretize} $\X$ if a morphism $\X(s<t)$ may fail to be an isomorphism only if there exists an $i$ such that $s<t_i\leq t$.
        \item We say that $\X$ is \textit{tame} if it admits a discretization. We denote by $(p\cC)^\tame \subseteq p\cC$ the full subcategory of $p\cC$ on the tame objects.
    \end{enumerate}
\end{definition}

The data of a tame object $\X\in (p\cC)^\tame$ amounts to a finite sequential diagram $$
    X_0 \to X_1 \to \cdots \to X_n
    $$
    in $\cC$, together with a sequence of non-negative reals $t_0 < t_1 < \ldots < t_n$. Given such sequences, the associated persistent object $\X$ is obtained by the left Kan extension $$
    \begin{tikzcd}
        {\{t_0  < \ldots < t_n\}} \ar[dr, "X"]\ar[d, "t"'] & \\
        {\R_+} \ar[r, "{\X=\Lan_t X}"'] & \cC
    \end{tikzcd}
    $$
    along the poset inclusion $\{t_0  < \ldots < t_n\} \subseteq \R_+$. The Kan extension is guaranteed to exist without assumptions on $\cC$, as is explained in \cite[Section~2.1]{giunti_2021}. This extension is constant over the half-open intervals $[t_{i}, t_{i+1})$; its components are given by $\X(t)=X_{t_i}$ when $t\in [t_i, t_{i+1})$, $\X(t)=X_{t_n}$ when $t\geq t_n$ and $\X(t)=0$ when $t<t_1$.

\subsection{Persistence Modules}\label{persistencesection}  A persistence module \cite{chazal2016structure} is a persistent object in the category $\mathsf{Vec}_\Q$ of $\Q$-vector spaces. In other words, a persistence module is a functor $\bM : \R_+ \to \mathsf{Vec}_\Q.$ A persistence module can equivalently be described as a graded module over a certain $\R_+$-graded ring, as follows. We denote by $\Q[X^{\R_+}]$ the ring of formal polynomials in one variable $X$ with \textit{non-negative real exponents}. Its elements are formal sums $$
\sum_{i=0}^n a_i X^{t_i}
$$
where $n\geq 0$ is an integer, the exponents $t_i$ are non-negative \textit{real numbers} and the coefficients $a_i$ belong to $\Q$. The multiplication on $\Q[X^{\R_+}]$ is given by a $\Q$-linear extension of $X^s\cdot X^t=X^{s+t}$. This ring admits a natural $\R_+$-grading 
$$
\Q[X^{\R_+}] \cong \bigoplus_{t\in \R_+} \Q X^t
$$
so that a persistent module $\V \in p\mathsf{Vec}_\Q$ can be described as a graded $\Q[X^{\R_+}]$-module. This identification can be promoted to an isomorphism of categories $$
p\mathsf{Vec}_\Q \cong \gr\Mod_{\Q[X^{\R_+}]}.
$$

\begin{definition}
    Let $s<t$ be non-negative real numbers. The \textit{interval module} $\bI_{[s,t)}$ is the persistence module with components given by:
    $$
    \bI_{[s,t)}(i) = \begin{cases} \Q & \text{ if } s\leq i < t \\ 0 & \text{ else} \end{cases}
    $$
    with structure maps $\bI_{[s,t)}(i\leq j)$ given by identities when $s\leq i \leq j < t$, and zero otherwise.
\end{definition}

Direct sums of persistence modules are constructed degree-wise on both objects and morphisms. 
Interval modules are the atomic units into which all (tame) persistence modules decompose, as stated below.

\begin{theorem}\cite[Thm. 1]{Webb_1985} \label{interval-decomposition}
    If $\V$ is a tame persistence module, then there exists a set of intervals $[s_\gamma, t_\gamma) \subseteq [0,\infty)$, $\gamma\in \Gamma,$ such that $\V$ splits as a direct sum of interval modules
    $$\V \cong \bigoplus_{\gamma \in \Gamma} \bI_{[s_\gamma,t_\gamma)}.$$
    Moreover, this decomposition is unique up to permutation of summands.
\end{theorem}

\begin{remark} The half-open intervals are well-suited to the study of tame objects, as $\bI_{[s,t)}$ is the left Kan extension of the sequence $\Q \to 0$ over the discretization $s < t$. The intervals $\bI_{[s,t]}, \bI_{(s,t]}$ and $\bI_{(s,t)}$ are, however, \textit{not} tame for $s < t$. If they were, they would permit a half-open interval decomposition by \ref{interval-decomposition}, which is impossible via a dimension argument. 
\end{remark}

\begin{comment} Note that there are many different assumptions and interval decomposition theorems used throughout the TDA literature. In this version, the indexing set $\Gamma$ can be arbitrarily large, and the theorem still holds. The original version in \cite[Thm.1]{Webb_1985} is stated in terms of $\Z$-graded $\mathbf{k}[x]$-modules which are bounded below. One translates the theorem into a statement about $[0,\infty)$ persistence modules via left Kan extension along an appropriate embedding map 
$$
\Z_{[-i,\infty)} \xrightarrow{+i} \N \to [0,\infty).
$$
\end{comment}
\begin{comment}
\begin{notation} \label{gamma_notation}
    For each interval $\bI_{[s_\gamma, t_\gamma)}$, we will denote the generator at time $s_\gamma \leq p < t_\gamma$ by the same symbol $\gamma_p$ as the index. This means that at time $p$ we have $$\V(p) = \bigoplus_{\gamma \in \Gamma} \bI_{[s_\gamma,t_\gamma)}(p) = \Q\{ \gamma_p \mid s_\gamma \leq p < t\}$$ and that
    $$\V(p \leq q)(\gamma_p) = \begin{cases} \gamma_q & q < t_\gamma\\ 0 & q \geq t_\gamma \end{cases}$$
\end{notation}
\end{comment}

\subsection{Persistent cochain complexes} 
A \textit{persistent cochain complex} is a persistent object in the category $\Ch_\Q$ of complexes. A persistent complex can equivalently be described as a complex of persistence modules. More precisely, there is an equivalence of categories 
$$
p\Ch_\Q \simeq \Ch(p\mathsf{Vec}_\Q) \simeq \Ch(\mathsf{grMod}_{\Q[X^{\R_+}]}),
$$
from which it follows that $p\Ch_\Q$ is an abelian category. As a category of presheaves valued in an abelian category, $p\Ch_\Q$ is itself an abelian category whose limits and colimits are computed pointwise. It is moreover a locally finitely presentable category, whence the small object argument \cite[Thn. 2.1.14]{Hovey2007} can be applied to \textit{any} set of maps.
\begin{notation} Let $\X \in p\Ch_\Q$ be a persistent cochain complex and $k \in \N$.
\begin{itemize}
    \item $\X^k, Z\X^k \in p\sVec$ are the persistence modules of $k$-cochains and $k$-cocycles respectively.
    \item $H^k \X \in p\sVec$ is the persistent $k$-th cohomology $\X$.
    \item If $x_s \in \X^k(s)$, then $x_t := \X^k(s \leq t)x_s$ for $s <t$.
\end{itemize}
\end{notation} 

\subsubsection{Interval spheres and disks} In this section we introduce interval spheres and disks, as first defined in \cite{giunti_2021}. These later serve as the cells in the model structure we define over $p\Ch_\Q$.

\begin{definition}[Interval spheres]\label{spheres} Let $0\leq s \leq t $ be non-negative real numbers and $k\in \N$.
\begin{enumerate}
    \item The \textit{interval sphere} $\Sp_{[s,t)}^k \in p\Ch_\Q$ is the left Kan-extension of the functor
    $$\{ s < t \} \to \Ch_\Q; \hspace{5em} s < t \mapsto S^k \hookrightarrow D^k,$$ along the inclusion of the subposet $\{ s < t\}$ into $\R_+$.
    \item The \textit{persistent disk} $\Di_s^k \in p\Ch_\Q$ is the left Kan-extension of the functor
    $$\{ s \} \to \Ch_\Q; \hspace{7em} s  \mapsto D^k,$$
    along the inclusion of  $\{ s \}$ into $\R_+$.
     \end{enumerate}

\end{definition} 

Note that $t$ may be infinite. There is a natural inclusion $\Sp_{[s,t)}^k \hookrightarrow \Di_s^k$ of persistent cochain complexes depicted in the following schematic
% https://q.uiver.app/#q=WzAsMjAsWzEsMywicyJdLFswLDIsImstMSJdLFswLDAsImsiXSxbMiwyLCJcXFEiXSxbMiwzLCJ0Il0sWzIsMCwiXFxRIl0sWzEsMCwiXFxRIl0sWzMsMCwiXFxRIl0sWzMsMiwiXFxRIl0sWzcsMywicyJdLFs3LDIsIlxcUSJdLFs3LDAsIlxcUSJdLFs4LDIsIlxcUSJdLFs4LDAsIlxcUSJdLFs1LDFdLFs2LDFdLFs0LDIsIlxcY2RvdHMiXSxbNCwwLCJcXGNkb3RzIl0sWzksMCwiXFxjZG90cyJdLFs5LDIsIlxcY2RvdHMiXSxbMyw1LCI9IiwyXSxbNiw1LCI9Il0sWzUsNywiPSJdLFszLDgsIj0iXSxbMTAsMTIsIj0iXSxbMTAsMTEsIj0iLDJdLFsxMSwxMywiPSJdLFsxNCwxNSwiIiwyLHsic3R5bGUiOnsidGFpbCI6eyJuYW1lIjoiaG9vayIsInNpZGUiOiJ0b3AifX19XSxbOCwxNl0sWzcsMTddLFs4LDcsIj0iLDJdLFsxMywxOF0sWzEyLDE5XSxbMTIsMTMsIj0iLDJdXQ==
\[\begin{tikzcd}[row sep = tiny]
	k & \Q & \Q & \Q & \cdots &&& \Q & \Q & \cdots \\
	&&&&& {} & {} \\
	{k-1} && \Q & \Q & \cdots &&& \Q & \Q & \cdots \\
	& s & t &&&&& s
	\arrow["{=}", from=1-2, to=1-3]
	\arrow["{=}", from=1-3, to=1-4]
	\arrow[from=1-4, to=1-5]
	\arrow["{=}", from=1-8, to=1-9]
	\arrow[from=1-9, to=1-10]
	\arrow[hook, from=2-6, to=2-7]
	\arrow["{=}"', from=3-3, to=1-3]
	\arrow["{=}", from=3-3, to=3-4]
	\arrow["{=}"', from=3-4, to=1-4]
	\arrow[from=3-4, to=3-5]
	\arrow["{=}"', from=3-8, to=1-8]
	\arrow["{=}", from=3-8, to=3-9]
	\arrow["{=}"', from=3-9, to=1-9]
	\arrow[from=3-9, to=3-10]
\end{tikzcd}\] where the vertical axis is degree and the horizontal is $\R_+$. As observed in \cite{giunti_2021}, maps out of persistent spheres and disks are particularly tractable given the natural isomorphisms of vector spaces $$
\Hom(\Di_s^k, \X)\cong \X^{k-1}(s) \:\:\:\:\quad\text{and}\quad\:\:\:\: \Hom(\Sp_{[s,t)}^k, \X) \cong Z\X^k(s)\times_{Z\X^k(t)}\X^{k-1}(t),
$$
valid for any persistent chain complex $\X$ and for each $k\geq 1$. This observation motivates the use of elements in the codomain to denote morphisms with source a persistent sphere or disk, as follows:
% https://q.uiver.app/#q=WzAsNCxbMCwwLCJcXFNwXntrKzF9X3tbcyx0KX0iXSxbMiwwLCJcXGJBIl0sWzQsMCwiXFxEaV57aysxfV9zIl0sWzYsMCwiXFxiQiJdLFswLDEsIih4X3MseV90KSJdLFsyLDMsInpfcyJdXQ==
\[\begin{tikzcd}
	{\Sp^{k+1}_{[s,t)}} && \bA && {\Di^{k+1}_s} && \bB
	\arrow["{(x_s,y_t)}", from=1-1, to=1-3]
	\arrow["{z_s}", from=1-5, to=1-7]
\end{tikzcd}\] where $dy_t = x_t \in \bA^{k+1}(t)$ and $z_s \in \bB^k(s)$.

\subsubsection{Interval spheres over $p\sVec$} More generally, interval spheres can be constructed over interval-decomposable persistence modules.
\begin{definition}
    Let $\V$ be a persistence module with interval decomposition
    $
    \V = \bigoplus_{\alpha} \bI_{[s_\alpha,t_\alpha)}.
    $
    The associated persistent spheres $\Sp^k(\V)$ and disks $\Di^k(\V)$ are defined for every $k\geq 0$ by 
    $$
    \Sp^k(\V)=\bigoplus_{\alpha} \Sp^k_{[s_\alpha, t_\alpha)}
    \:\:\:\:\:\:\:\:\:\:\:\:
    \Di^k(\V)=\bigoplus_{\alpha} \Di^k_{[s_\alpha, t_\alpha)}.
    $$
    This definition can be extended to \textit{graded} persistence modules $\V=\bigoplus_{k\geq 0} \V_k$, whenever each $\V_k$ is interval-decomposable, by setting 
    $$
    \Sp(\V)=\bigoplus_{k\geq 0} \Sp^k(\V_k)
    \:\:\:\:\:\:\:\:\:\:\:\:
    \Di(\V)=\bigoplus_{k\geq 0} \Di^k(\V_k).
    $$
    When $\V$ is concentrated in degree $k$, we find $\Sp(\V)=\Sp^k(\V)$ and similarly for disks. As before, there is a natural inclusion $\Sp(\V)\subseteq \Di(\V)$ induced by those of \Cref{spheres}.
\end{definition}

\begin{remark}
    It is important to emphasize that the notation $\Sp^k(\V)$ makes sense only when the persistence module $\V$ admits an interval decomposition. 
\end{remark}

\section{The interval-sphere model structure on $p\sCh$} 

In this section, we prove the existence of a new combinatorial, compactly generated model structure on $p\Ch_\Q$. Based on this framework, we can then interpret the interval surgery of\cite{hess2024celldecompositionspersistentminimal} as cell attachment in a transferred model structure on $p\CDGA_\Q$ in the next section, and the persistent minimal model of\cite{hess2024celldecompositionspersistentminimal} as a bona fide cell complex.

\subsection{The interval-sphere model structure} 
\label{modelstructurepchain}

Recall from \Cref{spheres} the definition of persistent $k$-spheres $\Sp^k_{[s,t)}$ and $k$-disks $\Di^k_{s}$. There are obvious inclusion maps 
$$
\Sp^k_{[s,t)}\to \Di^{k}_s 
\:\:\:\:\:\:\:\: \text{and} \:\:\:\:\:\:\:\ 
\Di^k_t \to \Di^{k}_s
$$
for any $s<t $ and $k\in \N$. When $t=\infty$, these inclusions become $$
\Sp^k_{[s,\infty)}\to \Di_s^k \:\:\:\:\:\:\:\: \text{and }\:\:\: 0 \to \Di_s^k.
$$
These inclusions play the role of generating cofibrations and trivial cofibrations for our model structure on $p\Ch_\Q$. It proves useful to treat the cases $t< \infty$ and $t=\infty$ separately. Define the following sets of maps in $p\Ch_\Q$ :
$$
\begin{aligned}
&\I_0 = \{ \:\: \Sp^k_{[s,t)}\to \Di^{k}_s \:\: | \: k\in \N, \:\: 0 \leq s<t<\infty  \:\} \\
&\I_\infty = \{ \:\: \Sp^k_{[s,\infty)}\to \Di^{k}_s \:\: | \: k\in \N, \:\: 0\leq s<\infty  \:\} \\
&\J_0 = \{ \:\: \Di^k_t \to \Di^{k}_s \:\: | \: k\in \N , \:\: 0\leq s< t < \infty \:\} \\
&\J_\infty = \{ \:\: 0 \to \Di^{k}_s \:\: | \: k\in \N , \:\: 0\leq s<  \infty \:\}.
\end{aligned}
$$
We let $\I=\I_0 \cup \I_\infty$ and $\J=\J_0 \cup \J_\infty$.  Note that there are evident inclusions $\Di_t^k \hookrightarrow \Sp_{[s,t)}^k$.  We adopt the convention that $\Di_s^0=0$.

Define $\W$ to be the class of maps $f:\X \to  \Y$ of which each component $f(i):\X(i)\to \Y(i)$ is a quasi-isomorphism.

Following classical terminology, we write $\Inj(\I)=\I^\pitchfork$ for the class of maps in $p\Ch_\Q$ that have the right lifting property with respect to maps of $\I$, and call them $\I$-injectives. 
Also, we write $\Cof(\I)=\:^\pitchfork\Inj(\I)= \:^\pitchfork(\I^\pitchfork)$ for the class of maps that have the left lifting property with respect to $\I$-injectives. We call those maps $\I$-cofibrations.

\begin{proposition}
The $\J_0$-injective maps are exactly those $f:\X \to \Y$ for which the induced maps $\X(s) \to \X(t) \times_{\Y(t)} \Y(s)$ are epimorphisms for every $0\leq s\leq t < \infty$.
\end{proposition}
\begin{proof}
If $s\leq t$, a commutative square of the form
$$
\begin{tikzcd}
    \Di_t^{k+1} \ar[d] \ar[r, "x_t"] & \X \ar[d, "f"] \\
    \Di_s^{k+1} \ar[r, "y_s"'] & \Y.
\end{tikzcd}
$$
is the data of a pair $(x_t, y_s)\in \X^{k}(t) \times \Y^{k}(s)$ such that $f(x_t)=y_t$. The set of all such squares is hence in bijection with the fiber product $\X^{k}(t)\times_{\Y^{k}(t)} \Y^{k}(s)$. Such a lifting problem admits a solution if and only if there exists $x_s\in \X^{k}(s)$ with $\X(s\leq t)(x_s)=x_t$. This means exactly that $(x_t, y_s)$ lies in the image of the induced map $\X^{k}(s) \to \X^{k}(t) \times_{\Y^{k}(t)} \Y^{k}(s)$. Hence all such lifting problems have a solution if and only if this map is surjective. Varying $k\geq 0$, we conclude that $\X(s) \to \X(t) \times_{\Y(t)} \Y(s)$ is an epimorphism.
\end{proof}

\begin{proposition}    
    The $\J_\infty$-injectives are exactly the epimorphisms.
\end{proposition}
\begin{proof}
    Given $s\geq 0$ and $k\in \N$, a commuting square 
    $$
    \begin{tikzcd}
    0 \ar[d] \ar[r] & \X \ar[d, "f"] \\
    \Di_s^{k+1} \ar[r, "y_s"] & \Y
    \end{tikzcd}
    $$
    admits a lift if and only if $y_s$ lies in the image of $f$. Any such square admits a lift if and only if $f:\X^k(s)\to \Y^k(s)$ is surjective. Varying $k$ and $s$, we find that $\J_\infty$-injectives are exactly the epimorphisms.
\end{proof}

%Note that the data of a map $\Sp_{[s,t]}^k \to \X$ amounts to the choice of a pair $(x_s,y_t)\in \X^k(s) \times \X^{k-1}(t)$ where $dy_t=x_t$. Similarly, a map $\Di_s^k \to \Y$ is the same thing as an element $u_s \in \Y^{k-1}(s)$. Commutativity of the above diagram translates as $f(x_s)=du_s$ and $f(y_t)=u_t$. Solving the lifting problem amounts to finding $y_s \in \X_s^{k-1}(s)$ such that $\X(s\leq t)(y_s)=y_t$ and $f(y_s)=u_s$. 

\begin{proposition}\label{cube}
The $\I_0$-injectives are those $f:\X \to \Y$ for which the induced maps of the cubes  
$$
\begin{tikzcd}[row sep=tiny, column sep=tiny]
& \Y^{k-1}(s)  \arrow[rr] \arrow[dd] & & Z\Y^k(s)  \arrow[dd] \\
\X^{k-1}(s) \arrow[ur]\arrow[rr, crossing over] \arrow[dd] & & Z\X^k(s) \arrow[ur] \\
& \Y^{k-1}(t)  \arrow[rr] & & Z\Y^{k}(t)  \\
\X^{k-1}(t) \arrow[ur]\arrow[rr] & & Z\X^k(t) \arrow[ur]\arrow[from=uu, crossing over]\\
\end{tikzcd}
$$
are epimorphisms for every $s\leq t$ and every $k\geq 1$.
\end{proposition}
\begin{remark}
The horizontal arrows in the cube are differentials. The vertical arrows are structure maps induced by the relation $s\leq t$ while the diagonal maps are components of $f$. The induced map of the cube above takes the form 
\begin{equation} \label{gapmap}
\X^{k-1}(s) \to \big(Z\X^k(s)\times_{Z\X^k(t)} \X^{k-1}(t)\big) \times_{\big(Z\Y^k(s)\times_{Z\Y^k(t)} \Y^{k-1}(t)\big)} \Y^{k-1}(s)
\end{equation}
whose components are respectively given by the induced map of the square $$
\begin{tikzcd}
    \X^{k-1}(s) \ar[d] \ar[r, "d"] & Z\X^{k}(s) \ar[d] \\
    \X^{k-1}(t) \ar[r, "d"'] & Z\X^k(t)
\end{tikzcd}
$$
and $f^{k-1}:\X^{k-1}(s)\to \Y^{k-1}(s)$. In elements, this induced map as described in \Cref{gapmap} sends an element $u_s\in \X^{k-1}(s)$ to the triple $(du_s, u_t, f(u_s))$. 
\end{remark}
\begin{proof}[Proof of \Cref{cube}]
    The set of commutative squares of the form 
    $$
    \begin{tikzcd}
        \Sp_{[s,t)}^{k} \ar[d] \ar[r, "{(x_s,\: u_t)}"] & \X \ar[d, "f"] \\
        \Di_s^k \ar[r, "y_s"'] & \Y.
    \end{tikzcd}
    $$
    is in bijection with the set of triples $(x_s, u_t, y_s)\in Z\X^k\times \X^{k-1}\times \Y^{k-1}$ such that $ du_t=x_t$, \linebreak $f(x_s)=dy_s$ and $f(u_t)=y_t$. This set is in one-to-one correspondance with the iterated fiber product $$
    \big(Z\X^k(s)\times_{Z\X^k(t)} \X^{k-1}(t)\big) \times_{\big(Z\Y^k(s)\times_{Z\Y^k(t)} \Y^{k-1}(t)\big)} \Y^{k-1}(s).
    $$
    Such a lifting problem admits a solution if and only if there exists $u_s \in \X^{k-1}$ such that \linebreak $\X(s\leq t)(u_s)=u_t$, $f(u_s)=y_s$ and $du_s=x_s$. This precisely means that the induced map (\ref{gapmap}) sends $u_s$ to the triple $(x_s, u_t, y_s)$.
\end{proof}

\begin{proposition}\label{qiso}
The $\I_\infty$-injectives are exactly the pointwise surjective quasi-isomorphisms.
\end{proposition}
\begin{proof}
Let $f:\X \to \Y$ be an $\I_\infty$-injective. We wish to show that $Hf:H\X(s)\to H\Y(s)$ is an isomorphism for every $s\geq0$. Let $x_s \in \X^k$ be a cocycle with $f(x_s)=dy_s$ for some $y_s\in \Y^{k-1}$. This data corresponds to a square 
\begin{equation}\label{squareinf}
\begin{tikzcd}
        \Sp_{[s,\infty)}^{k} \ar[d] \ar[r, "x_s"] & \X \ar[d, "f"] \\
        \Di_s^k\ar[ur, dashed] \ar[r, "y_s"'] & \Y
    \end{tikzcd}
\end{equation}
where the map $\Sp^k_{[s,\infty]}\to \X$ corresponds to the choice of $x_s \in \X^k$ and the map $\Di_s^k\to \Y$ to $y_s\in \Y^{k-1}$. The square commutes since $f(x_s)=dy_s$. Existence of the lift means there exists $u_s \in \X^{k-1}$ with $du_s=x_s$. This shows that $Hf$ is injective.

For surjectivity, let $y_s\in \Y^k$ be a cocycle. We obtain the following square of solid arrows
$$
\begin{tikzcd}
        \Sp_{[s,\infty)}^{k+1} \ar[d] \ar[r, "0"] & \X \ar[d, "f"] \\
        \Di_s^{k+1}\ar[ur, dashed] \ar[r, "y_s"'] & \Y
    \end{tikzcd}
$$
where the bottom map corresponds to the choice of $y_s$. Commutativity of the square follows from $dy_s=0$. Existence of the lift amounts to finding $x_s\in \X^k$ with $dx_s=0$ and $f(x_s)=y_s$. Since the choice of $y_s$ was arbitrary, this proves that $Hf$ is surjective.

Now we prove that $f$ itself is surjective. To that end, take $y_s \in \Y^k$. Since $dy_s$ is a cocycle, there exists $x_s\in Z\X^{k+1}$ such that $f(x_s)=dy_s$. The pair $(x_s, y_s)$ defines a diagram of form \Cref{squareinf}, so there exists $u_s\in \X^{k-1}$ with $f(u_s)=y_s$.

For the converse, assume $f$ is surjective and $Hf$ an isomorphism. Under these assumptions there is a short exact sequence $$
0 \to \K \to \X \overset{f}{\to} \Y \to 0
$$
with $H\K=0$. Consider a square of solid arrows of the form Diagram \ref{squareinf} above. Choose $u_s\in \X^{k-1}$ with $f(u_s)=y_s$. Then $f(du_s)=dy_s=f(x_s)$ so $du_s-x_s \in \K$ hence there exists $z_s\in \K^k$ such that $dz_s=du_s-x_s$. Letting $v_s=u_s-z_s$ we find $f(v_s)=f(u_s)$ and $z_t=0$, which means that $v_s$ defines a lift in diagram \Cref{squareinf}.
\end{proof}

%\begin{remark}\label{inftyinj}
%    The proof of \Cref{qiso} shows in fact that the pointwise surjective quasi-isomorphisms are \textit{exactly} the maps that have the right lifting property with respect to inclusions $\Sp_{[s,\infty]}^k \to \Di_s^k$ for every $k\in \N$ and $s\geq 0$. Writing $\I_\infty \subseteq \I$ for the set of all such inclusions, we find $ \W=(\I_\infty)^\pitchfork
%    $.
%\end{remark}

\begin{proposition}\label{IJinj}
The $\I$-injectives are also $\J$-injectives.
\end{proposition}
\begin{proof}
    Let $f:\X \to \Y$ be an $\I$-injective and consider a lifting problem 
\begin{equation}\label{liftproblem}
    \begin{tikzcd}
        \Di^{k+1}_t \ar[d] \ar[r, "x_t"] & \X \ar[d, "f"] \\
        \Di_s^{k+1} \ar[r, "y_s"'] & \Y
    \end{tikzcd}
\end{equation}
in which $k\geq 0$ and $s \leq t$, corresponding to a pair $(x_t, y_s)\in \X^k \times \Y^k$ such that $f(x_t)=y_t$. The pair $(0_s,dx_t)\in Z\X^{k+2}\times \X^{k+1}$ corresponds to a map $\Sp_{[s,t)}^{k+2} \to \X$, while $dy_t$ defines a map $\Di_t^{k+2}\to \X$. That $f(dx_t)=dy_t$ ensures commutativity of the solid square
$$
\begin{tikzcd}
        \Sp_{[s,t)}^{k+2} \ar[d] \ar[r, "{(0_s,\:dx_t)}"] & \X \ar[d, "f"] \\
        \Di_s^{k+2}\ar[ur, dashed] \ar[r, "dy_s"'] & \Y.
    \end{tikzcd}
$$
Existence of a lift guarantees that we can find $z_s \in \X^{k+1}$ with $z_t=dx_t$ and $f(z_s)=dy_s$. Since $dz_s=0$, the pair $(z_s, x_t)\in Z\X^{k+1}\times \X^k$ defines a map $\Sp_{[s,t)}^{k+1}\to \X$ fitting in the square of solid arrows
$$
    \begin{tikzcd}
        \Sp_{[s,t)}^{k+1} \ar[d] \ar[r, "{(z_s,\:x_t)}"] & \X \ar[d, "f"] \\
        \Di_s^{k+1}\ar[ur, dashed] \ar[r, "y_s"'] & \Y.
    \end{tikzcd}
$$
Any lift $x_s \in \X^k$ solves the original lifting problem because $\X(s\leq t)(x_s)=x_t$.

For the case $k=-1$, the top map in Diagram \ref{liftproblem} consists of a pair of cocycles $(x_t, y_s)\in Z\X^0 \times Z\Y^0$ with $f(x_t)=y_t$. The pair $(0, x_t)$ specifies a map $\Sp_{[s,t)}^1\to \X$ rendering the solid square on the left below commutative.
$$
    \begin{tikzcd}
        \Sp_{[s,t)}^{1} \ar[d] \ar[r, "{(0,\:x_t)}"] & \X \ar[d, "f"] \\
        \Di_s^{1}\ar[ur, dashed, "x_s"] \ar[r, "y_s"'] & \Y
    \end{tikzcd}
    \:\:\:\:\:\:\:\:\:\:\:\:\:\:\:\:\:\:\:\:\:\:\:\:\:\:\:\:\:\:\:\:\:\:\:\:
    \begin{tikzcd}
        \Di_t^{0} \ar[d] \ar[r, "x_t"] & \X \ar[d, "f"] \\
        \Di_s^{0}\ar[ur, dashed, "x_s"] \ar[r, "y_s"'] & \Y
    \end{tikzcd}
$$
The lift $x_s \in Z\X^0$ on the left solves the lifting problem on the right since $f(x_s)=y_s$ and $\X(s\leq t)(x_s)=x_t$.
\end{proof}

\begin{corollary}\label{JcofIcof}
The maps in $\J$ are both $\I$-cofibrations and weak equivalences.
\end{corollary}
\begin{proof}
    It is immediate that $\J \subseteq \W$. By \Cref{IJinj} we have $\I^\pitchfork \subseteq \J^\pitchfork$ from which it follows that $\J\subseteq \: ^\pitchfork (\J ^\pitchfork) \subseteq \: ^\pitchfork (\I ^\pitchfork)$.
\end{proof}

\begin{proposition}\label{cofcheck}
    $\Cof(\J)\subseteq \Cof(\I)\cap \W$.
\end{proposition}

\begin{proof}
   \Cref{IJinj} implies that $\Inj(\I) \subseteq \Inj(\J)$, from which it follows that $\Cof(\J) \subseteq \Cof(\I)$. It remains to show that $\Cof(\J) \subseteq \W$. Note that the class $\W$ is closed under retracts and transfinite composition. This last statement follows from the fact that filtered colimits are exact (thus commute with cohomology) in $\Ch_Q$, hence also in $p\Ch_\Q$, since colimits are computed pointwise. The proof of the claim thus reduces to showing that pushouts of maps in $\J$ are weak equivalences. 
   
   Any pushout square 
    $$
    \begin{tikzcd}
        \Di_t^{k} \arrow[dr, phantom, "\scalebox{1}{$\ulcorner$}" , very near end, color=black] \ar[d] \ar[r, "{x_t}"] & \X \ar[d, "f"] \\
        \Di_s^{k} \ar[r] & \X'.
    \end{tikzcd}
    $$
    induces an isomorphism on cofibers $\X/\X' \cong \Di^k_s / \Di^k_t$. Since this complex is acyclic, the long exact sequence in cohomology shows that $f$ is a pointwise quasi-isomorphism.
\end{proof}

\begin{proposition}\label{Iinj_JinjW}
    The $\I$-injectives are exactly those $\J$-injectives that are also pointwise quasi-isomorphisms.
\end{proposition}
\begin{proof}
    It follows from \cref{qiso} and \ref{IJinj} that $\I^\pitchfork \subseteq \J^\pitchfork \cap \W$, where we use $\bI_\infty \subset \bI$ for the weak equivalence part. For the reverse inclusion, let $f:\X \to \Y$ be both $\J$-injective and a pointwise quasi-isomorphism. Being a $\J$-injective implies that $f$ is surjective. \Cref{qiso} implies that $f$ is $\I_\infty$-injective because $f\in \W$. 
    
    It remains to show that $f$ is also $\bI_0$ injective. There is a short exact sequence $$
    0 \to \K \to \X \overset{f}{\to} \Y \to 0
    $$
    with $H\K=0$. Let $s\leq t < \infty$ and $k\geq 0$. Assume given a commutative square 
    \begin{equation}\label{Ilift}
    \begin{tikzcd}
        \Sp_{[s,t)}^{k+1} \ar[d] \ar[r, "{(x_s,\:u_t)}"] & \X \ar[d, "f"] \\
        \Di_s^{k+1} \ar[r, "y_s"'] & \Y
    \end{tikzcd}
    \end{equation}
    for which a lift is to be found. The pair $(u_t, y_s)$ fits in the square of solid arrows $$
    \begin{tikzcd}
        \Di_t^{k+1} \ar[d] \ar[r, "{u_t}"] & \X \ar[d, "f"] \\
        \Di_s^{k+1} \ar[ur, dashed, "u_s"]\ar[r, "y_s"'] & \Y
    \end{tikzcd}
    $$
    where a lift $u_s \in \X^k$ exists since $f$ is $\J$-injective. Now $f(du_s)=df(u_s)=dy_s=f(x_s)$, so $du_s-x_s\in Z\K^{k+1}$ is a cocycle. Since $H\K =0$ we can find $z_s\in \K^{k}$ with $dz_s=x_s-du_s$. At the later time $t$, we find $dz_t=x_t-du_t=0$, so $z_t\in Z\K^k$. 

    In the case where $k=0$, we have $Z\K^0=H\K^0 = 0$, which implies that $z_t=0$. The element $v_s=u_s+z_s$ thus provides a lift to the square in (\ref{Ilift}), since both $dv_s=du_s+dz_s=x_s$ and $f(v_s)=f(u_s)=y_s$.
    
    Assume now that $k\geq 1$. Since $f$ is a quasi-isomorphism, there must be $w_t \in \K^{k-1}$ with $dw_t=z_t$. Being a pullback of a $\J$-injective, the unique map $\K \to 0$ is also a $\J$-injective. Applying $\J$-injectivity to the square $$
    \begin{tikzcd}
        \Di_t^{k} \ar[d] \ar[r, "{w_t}"] & \K \ar[d] \\
        \Di_s^{k} \ar[ur, dashed, "w_s"]\ar[r] & 0,
    \end{tikzcd}
    $$
    we find $w_s \in \K^{k-1}$ with $\K(s\leq t)(w_s)=w_t$, satisfying $d(z_s-dw_s)=x_s-du_s$ and $z_t-dw_t=0$. Setting $v_s=u_s+z_s-dw_s$, we find that $f(v_s)=f(u_s)=y_s$ and $dv_s=dz_s+du_s=x_s$. Finally $v_t=z_t-dw_t+u_t=u_t$, meaning that $v_s$ defines a lift for (\ref{Ilift}) as desired.

    Lastly, we consider (\ref{Ilift}) in the case $k=-1$. In this situation, the lifting problem corresponds to a cocycle $x_s\in Z\X^0$ with $f(x_s)=0$. Since $f$ is a quasi-isomorphism, it induces an isomorphism $Z\X^0 \cong Z\Y^0$ so that $x_s=0$, and the lifting problem is solved. 
\end{proof}

\begin{theorem}[Interval-Sphere Model Structure]
There is a cofibrantly generated model structure on $p\Ch_\Q$ for which $\W$ is the class of weak equivalences and the sets $\I, \J$ are those of the generating cofibrations and generating trivial cofibrations, respectively.
\end{theorem}
\begin{proof}
    We use the recognition theorem of D. Kan, recorded here as \cite[Theorem~11.3.1]{HH}. The following conditions need to be checked in order for the theorem to apply.
    \begin{enumerate}
        \item\label{1} $\W$ is closed under retracts and satisfies 2-out-of-3.
        \item\label{2} Persistent spheres and disks are $\kappa$-compact in $p\Ch_\Q$ for some cardinal $\kappa$.
        \item\label{3} $\Inj(\I)= \Inj(\J)\cap \W$.
        \item\label{4} $\Cof(\J)\subseteq \Cof(\I)\cap \W$.
    \end{enumerate}
    Condition \ref{1} is immediate, while \ref{2} follows directly from $p\Ch_\Q$ being a locally finitely presentable category. Condition \ref{3} was proved as \Cref{Iinj_JinjW}, and Condition \ref{4} is \Cref{cofcheck}.
\end{proof}

\subsection{Comparison with the projective model structure}\label{not-projective}

As a category of diagrams, $p\Ch_\Q=\Fun(\R_+, \Ch_\Q)$ admits two other natural model structures, namely the projective and injective model structures. Both these structures exist and are cofibrantly generated. This is proved for instance in \cite[Theorem~11.6.1]{HH} for the case of the projective model structure, and in \cite[Theorem~2.2]{hoveysheaf} for the injective model structure, as a result of the fact that $p\Ch_\Q=\Ch(\Fun(\R_+, \mathsf{Vec}_\Q))$ is a category of complexes in a Grothendieck abelian category.

Through an example, we demonstrate that the interval sphere model structure is different from the projective model structure on $p\Ch_\Q$. More precisely, we exhibit  a map of persistent cochain complexes that is pointwise surjective, though not a fibration in the interval sphere model structure. We show later (\Cref{non-injective-model-example}) that the interval-sphere model structure is not the injective model structure after introducing more theory.

Given non-negative real numbers $s<t<u$, denote by $q:\X\to\Y$ the quotient map $\Di^k_s \to \Di^k_s/\Di^k_t$. It is pointwise surjective, but we claim it is not a fibration in the interval-sphere model structure. The element $1_s\in \Y^k(s)=\Q$ satisfies $\Y(s\leq u)(1_s)=0=f(0_{u})$ where $0_{u}\in \X^k(u)$. However the pair $(1_s, 0_u)\in \Y(s)\times_{\Y(u)}\X(u)$ admits no pre-image under the induced map $\X(s)\to \Y(s)\times_{\Y(u)}\X(u)$, since $f(s)^{-1}(1_s)=\{1_s\}$ and $\X(s<u)(1_s)=1_u\ne 0_u$. In other words, no lift can exist in the diagram
\begin{center}
    \begin{tikzcd}
        \Di^{k}_{t} \ar[r, "0"] \ar[d, hookrightarrow] & \Di^{k}_{s} \ar[d, "q"] \\
        \Di^k_{s} \ar[r, "q"] & \Di^{k}_{s}/ \Di^{k}_{t}.
    \end{tikzcd}
\end{center}

\section{Cofibrancy and local compactness}  In this section, we define local compactness as a measure of being locally interval decomposable. We prove that this is a necessary condition for cofibrancy in the interal-sphere model structure, distinguishing it from the usual injective model structure. 

\subsection{Compactness in categories of persistent objects}  In this section we study the compact objects in various categories of persistent objects. Recall that an object $X$ in a category $\cC$ is \textit{compact} if the functor $\Hom(X,-)$ preserves filtered colimits. We denote by $\cC^\omega \subseteq \cC$ the full subcategory on the compact objects. We use the terms \textit{finite type} and \textit{finite} for spaces, (graded) vector spaces, complexes and CDGAs. For each of these categories, we say a (co)persistent object $\X$ is of finite type (resp. finite) if its components $\X(t)$ are finite type (resp. finite) for all indices $t$.

\begin{proposition}\label{compactpch}
    The compact objects in $p\mathsf{Vec}_\Q$ are exactly the tame modules of finite type.
\end{proposition}
\begin{proof}
    Let us write $R=\Q[X^{\R_+}]$ for the `polynomial' ring of \Cref{persistencesection}. In view of the identification $p\mathsf{Vec}_\Q\cong \gr\Mod_{R}$, the compact objects in $p\mathsf{Vec}_\Q$ are exactly the graded $R$-modules of finite presentation. Given a persistence module $\V$, such a finite presentation amounts to a short exact sequence $$
    0 \to \bigoplus_{j\in J} X^{t_j}R \to \bigoplus_{i\in I} X^{s_i} R \to \V \to 0
    $$
    where $I$ and $J$ are finite sets. This shows that $\V$ is a quotient of tame modules of finite type, hence is itself tame and of finite type. Conversely, if $\V$ is tame and of finite type, it admits an interval decomposition 
    $
    \V\cong \bigoplus_{i\in I} \I_{[s_i, t_i)}
    $
    consisting of only a finite number of intervals. Such a decomposition implies the existence of a short exact sequence
    $$
    0 \to \bigoplus_{i\in I} \I_{[t_i, \infty)} \to \bigoplus_{i\in I} \I_{[s_i, \infty)} \to \V \to 0,
    $$
    which translates to a finite presentation of $\V$ in view of the isomorphism of persistence modules $X^tR \cong \I_{[t,\infty)}$.
\end{proof}
% Recall from \Cref{notation} that we say that a complex $X\in \Ch_\Q$ is of \textit{finite type} if $\bigoplus_{k\geq 0} X^k$ is a finite dimensional vector space. Equivalently, $X$ is of finite type if it is a bounded complex of finite dimensional vector spaces. It is well known that complexes of finite type are the compact objects in $\Ch_\Q$. Similarly, we say that a space $X\in \sSet$ is of \textit{finite type} if it has a finite number of non-degenerate simplices.
\begin{proposition}
    The compact objects of $p\Ch_\Q$ are exactly the tame finite complexes.
\end{proposition}

\begin{proof}
    In view of the identification $p\Ch_\Q \cong \Ch(\gr\Mod_{\Q[X^{\R_+}]})$, the compact objects in $p\Ch_\Q$ are the bounded complexes of finitely presented $\R_+$-graded $\Q[X^{\R_+}]$-modules. By \Cref{compactpch}, these are the bounded complexes of tame persistence modules of finite type.
\end{proof}

\begin{proposition} \label{compact-p-complexes}
    The compact objects in copersistent spaces are precisely the tame finite spaces.
\end{proposition}
\begin{proof}
    The category of copersistent spaces can equivalently be described as the presehaf category
    $$
    \Fun(\R_+^\op \times \DDelta^\op, \Set).
    $$
    The compact objects in a presheaf category are precisely the finite colimits of representables. Given $n\in \N$ and $t\geq 0$, the representable $\Hom_{\R_+ \times \DDelta}(-,t\times [n])$, when seen as a copersistent space, can be described as follows
    $$
    (\mathbb{1}_{\leq t} \times \Delta^n) : s \mapsto \begin{cases}
        \Delta^n & \text{when $s \leq t$}, \\
        \varnothing & \text{otherwise,}
    \end{cases}
    $$
    with structure maps being either identities or the trivial map $\varnothing \to \Delta^n$. As such, it is tame and finite. Accordingly, any compact copersistent space $\X$ is tame and finite, since it admits a surjection
    $$
    \coprod_{1\leq i \leq n} (\mathbb{1}_{\leq t_i} \times \Delta^{n_i}) \twoheadrightarrow \X
    $$
    from a tame and finite copersistent space.
\end{proof}

\subsection{Compactly generated model structure} We prove now that the interval sphere model structure is compactly generated. The payoff is that $\bI$-cell complexes can be expressed as countable sequences of (potentially uncountably many) cell attachments. We present the basic definitions of compactly generated model categories and refer to \cite{May_Ponto_2011} for a complete reference. 

\begin{definition} 
    For an ordinal $\lambda$ and an object $\X \in p\Ch_\Q$, a \textit{relative $\bI$-cell $\lambda$-complex} under $\X$ is a map $f : \X \to \Y$ that is a transfinite composite of a $\lambda$-sequence such that $\Y_0 = \X$ and for $\alpha +1 < \lambda$, the object $\Y_{\alpha+1}$ is obtained as a pushout
    % https://q.uiver.app/#q=WzAsNCxbMCwwLCJcXGJpZ29wbHVzX3tcXGdhbW1hIFxcaW4gXFxHYW1tYV9cXGFscGhhfSBcXFNwXntrKzF9X3tbc19cXGdhbW1hLHRfXFxnYW1tYSl9Il0sWzIsMCwiXFxZX1xcYWxwaGEiXSxbMCwyLCJcXGJpZ29wbHVzX3tcXGdhbW1hIFxcaW4gXFxHYW1tYV9cXGFscGhhfSBcXERpXntrKzF9X3tzX1xcZ2FtbWF9Il0sWzIsMiwiXFxZX3tcXGFscGhhKzF9Il0sWzAsMl0sWzIsM10sWzEsM10sWzAsMV0sWzMsMCwiIiwxLHsic3R5bGUiOnsibmFtZSI6ImNvcm5lciJ9fV1d
\[\begin{tikzcd}[sep=small]
	{\bigoplus_{\gamma \in \Gamma_\alpha} \Sp^{k+1}_{[s_\gamma,t_\gamma)}} && {\Y_\alpha} \\
	\\
	{\bigoplus_{\gamma \in \Gamma_\alpha} \Di^{k+1}_{s_\gamma}} && {\Y_{\alpha+1}}
	\arrow[from=1-1, to=3-1]
	\arrow[from=3-1, to=3-3]
	\arrow[from=1-3, to=3-3]
	\arrow[from=1-1, to=1-3]
	\arrow["\lrcorner"{anchor=center, pos=0.125, rotate=180}, draw=none, from=3-3, to=1-1]
\end{tikzcd}\] over a set $\Gamma_\alpha$ of arbitrary cardinality. A \textit{sequential} relative $\bI$-cell complex is a relative $\bI$-cell $\omega$-complex where $\omega$ is a countable ordinal. 
\end{definition} 

A morphism $f : \X \to \Y$ is a relative $\bI$-cell complex $\X$ if there exists some ordinal $\lambda$ such that $\X$ is a relative $\bI$-cell $\lambda$-complex. The set of relative $\bI$-cell complexes is denoted $\text{cell}(\bI)$. In a cofibrantly generated model category with cofibrant generators $\bI$, the set of cofibrations consists of retracts of $\text{cell}(\bI)$.

\begin{definition}
    An object $\bU \in p\Ch_\Q$ is \textit{compact with respect to $\bI$} if for every relative $\bI$-cell $\omega$-complex $f : \X \to \Y = \colim \Y_i$, the map
    $$\colim_i \Hom(\bU,\Y_i) \to \Hom(\bU,\Y)$$ is a bijection. We say a cofibrantly generated model category with cofibrant generators $\bI$ is \textit{compactly generated} if all domain objects in $\bI$ are compact with respect to relative sequential $\bI$-cell complexes. 
\end{definition}

In a compactly generated model category, the set of cofibrations is equal to the set of retracts of sequential relative $\bI$-cell complexes \cite[15.2.1]{May_Ponto_2011}.

\begin{proposition}\label{cofib-gen}
    The interval-sphere model structure is compactly generated.
\end{proposition}

\begin{proof}
    The set of source objects in $\bI$ are the interval spheres $\Sp_{[s,t)}^k$. By \Cref{compact-p-complexes}, these are compact with respect to all filtered colimits, and thus compact with respect to sequential relative $\bI$-cell complexes.
\end{proof}

\subsubsection{Locally compact persistence modules} Local compactness is the central concept that proves to be a necessary condition for $\bI$-cell complexes. We first define and study local compactness at the level of persistence modules.

\begin{definition}
    Let $\V \in p\mathsf{Vec}_\Q$. 
    \begin{enumerate}
        \item A \textit{compact neighbourhood} of $v \in \V$ is a compact persistence module $\bU_v$ satisfying \linebreak $v \in \bU_v \subseteq \V$.
        \item A persistence module $\V$ is \textit{locally compact} if every point $v \in \V$ admits a compact neighbourhood.
    \end{enumerate} 
\end{definition}

Since $\bU_v$ is compact if and only if it admits a finite interval decomposition, we can think of local compactness at $v$ as local interval decomposability around $v$. For a set of submodules $\{ \bU_\alpha \subseteq \V \},$ the \textit{union} $\bigcup_\alpha \bU_\alpha$ is the image of the map $\bigoplus_\alpha \bU_\alpha \twoheadrightarrow \V$ under the maps induced by canonical inclusions $\bU_\alpha \hookrightarrow \V$. We call the set $\{ \bU_\alpha \}$ a \textit{cover} if $\bigcup_\alpha \bU_\alpha = \V$ or equivalently if the map $\bigoplus_\alpha \bU_\alpha \twoheadrightarrow \V$ is surjective. It is not hard to see that ocal compactness is equivalent to the existence of a cover $\{ \bU_\alpha \subseteq \V \mid \bU_\alpha \in (p\mathsf{Vec}_\Q)^\omega \}$ of $\V$ by compact sub-modules, and we use the two definitions interchangeably.

\begin{proposition}
    The class of locally compact persistence modules is closed under arbitrary direct sums. 
\end{proposition}

\begin{proof}
    If $\V = \oplus_\beta \V_\beta$ is a direct sum of locally compact persistence modules with compact covers $\oplus_{\alpha \in A_\beta} \bU_{\alpha, \beta} \twoheadrightarrow \V_\beta$, then 
    $$\bigoplus_\beta \bigoplus_{\alpha \in A_\beta} \bU_{\beta,\alpha} \twoheadrightarrow \bigoplus_\beta \V_\beta$$ is a compact cover for $\V$.
\end{proof}

There is a large, easily identifiable class of persistence modules that are not locally compact. An element $v_r \in \V(r)$ is a \textit{right-closed} point of $\V$ if $\V(r + \epsilon)(v_r) = 0$ for all $\epsilon > 0$.

\begin{lemma}
    If $\V \in p\mathsf{Vec}_\Q$ contains a right-closed point, then $\V$ is not locally compact.
\end{lemma}

\begin{proof}
    Let $v_r \in \V(r)$ be a right-closed point. Suppose that there exists a compact neighbourhood $i : \bU_{v_r} = \oplus_{i=0}^n \bI_{[s_i,t_i)} \hookrightarrow \V$ containing $v_r = i(u_r)$. Pick $\epsilon < \min \{ t_i - r \mid 0 \leq i \leq n\}$. Then $v_{r+\epsilon} = i(u_{r+\epsilon}) = 0$ but $u_{r+\epsilon} \neq 0$, contradicting the injectivity of $i$.
\end{proof}

\begin{lemma} \label{pmod-preserving-lcocal-compactness}
    If $\V \in p\mathsf{Vec}_\Q$ is locally compact, then the pushout
    % https://q.uiver.app/#q=WzAsNCxbMCwwLCJcXGJpZ29wbHVzX3tcXGdhbW1hIFxcaW4gXFxHYW1tYX0gXFxiSV97W3RfXFxnYW1tYSxcXGluZnR5KX0gIl0sWzAsMiwiXFxiaWdvcGx1c197XFxnYW1tYSBcXGluIFxcR2FtbWF9IFxcYklfe1tzX1xcZ2FtbWEsXFxpbmZ0eSl9Il0sWzIsMCwiXFxWIl0sWzIsMiwiXFxvdmVybGluZXtcXFZ9Il0sWzAsMiwidl90Il0sWzEsMywiXFxnYW1tYV97c19cXGdhbW1hfSJdLFsyLDMsIiIsMSx7InN0eWxlIjp7InRhaWwiOnsibmFtZSI6Imhvb2siLCJzaWRlIjoidG9wIn19fV0sWzAsMSwiIiwxLHsic3R5bGUiOnsidGFpbCI6eyJuYW1lIjoiaG9vayIsInNpZGUiOiJ0b3AifX19XV0=
\[\begin{tikzcd}[sep=small]
	{\bigoplus_{\gamma \in \Gamma} \bI_{[t_\gamma,\infty)} } && \V \\
	\\
	{\bigoplus_{\gamma \in \Gamma} \bI_{[s_\gamma,\infty)}} && {\overline{\V}}
	\arrow["{v_{t_\gamma}}", from=1-1, to=1-3]
	\arrow["{\gamma_{s_\gamma}}", from=3-1, to=3-3]
	\arrow["i", hook, from=1-3, to=3-3]
	\arrow[hook, from=1-1, to=3-1]
\end{tikzcd}\] over an arbitrary indexing set $\Gamma$ is locally compact for all $\{ s_\gamma \leq t_\gamma \mid \gamma \in \Gamma\}$.
\end{lemma}

\begin{proof}
    Elements in $\overline{\V}$ consist of finite sums of homogenous elements
    $$ v_r \in \overline{\V}(r) = \V(r) \oplus (\overline{\V} / i\V)(r) = \V(r) \oplus \Q \langle \gamma_r \mid s_\gamma \leq r < t_\gamma \rangle ,$$ where we have used the fact that the exact sequence
    $0 \to \V \to \overline{\V} \to \overline{\V}/i\V \to 0$ splits pointwise at each $r \in \R_+$. Finite sums of locally compact objects are locally compact, since a finite union of compact neighbourhoods is compact. Since elements of $i(\V)$ admit compact neighborhoods by assumption, it remains only to check that the elements in each direct summand of $\overline{\V}/i\V(r) \cong \oplus_\gamma \bI_{[s_\gamma,t_\gamma)}(r)$ do as well. 
    
    We factorise the pushout diagram to get
    % https://q.uiver.app/#q=WzAsNyxbMCwwLCJcXGJJX3tbdF9cXGdhbW1hLFxcaW5mdHkpfSJdLFsyLDAsIlxcYlVfe3Zfe3RfXFxnYW1tYX19Il0sWzIsMiwiXFxiVV97XFxnYW1tYV97c19cXGdhbW1hfX0iXSxbMCwyLCJcXGJJX3tbc19cXGdhbW1hLFxcaW5mdHkpfSJdLFs0LDAsIlxcViJdLFs0LDIsIlxcVl9cXGdhbW1hIl0sWzYsMSwiXFxvdmVybGluZXtcXFZ9Il0sWzAsMSwidl97dF9cXGdhbW1hfSJdLFswLDMsIiIsMix7InN0eWxlIjp7InRhaWwiOnsibmFtZSI6Imhvb2siLCJzaWRlIjoidG9wIn19fV0sWzMsMiwiXFxnYW1tYV97c19cXGdhbW1hfSIsMl0sWzIsMCwiIiwxLHsic3R5bGUiOnsibmFtZSI6ImNvcm5lciJ9fV0sWzEsMiwiIiwwLHsic3R5bGUiOnsidGFpbCI6eyJuYW1lIjoiaG9vayIsInNpZGUiOiJ0b3AifX19XSxbMSw0LCIiLDEseyJzdHlsZSI6eyJ0YWlsIjp7Im5hbWUiOiJob29rIiwic2lkZSI6InRvcCJ9fX1dLFsyLDUsIiIsMSx7InN0eWxlIjp7InRhaWwiOnsibmFtZSI6Imhvb2siLCJzaWRlIjoidG9wIn19fV0sWzUsNiwiIiwxLHsic3R5bGUiOnsidGFpbCI6eyJuYW1lIjoiaG9vayIsInNpZGUiOiJ0b3AifX19XSxbNCw1LCIiLDEseyJzdHlsZSI6eyJ0YWlsIjp7Im5hbWUiOiJob29rIiwic2lkZSI6InRvcCJ9fX1dLFs0LDYsIiIsMSx7InN0eWxlIjp7InRhaWwiOnsibmFtZSI6Imhvb2siLCJzaWRlIjoidG9wIn19fV0sWzUsMSwiIiwxLHsic3R5bGUiOnsibmFtZSI6ImNvcm5lciJ9fV1d
\[\begin{tikzcd}[row sep=tiny, column sep = small]
	{\bI_{[t_\gamma,\infty)}} && {\bU_{v_{t_\gamma}}} && \V \\
	&&&&&& {\overline{\V}} \\
	{\bI_{[s_\gamma,\infty)}} && {\bU_{\gamma_{s_\gamma}}} && {\V_\gamma}
	\arrow["{v_{t_\gamma}}", from=1-1, to=1-3]
	\arrow[hook, from=1-1, to=3-1]
	\arrow["{\gamma_{s_\gamma}}"', from=3-1, to=3-3]
	\arrow["\lrcorner"{anchor=center, pos=0.125, rotate=180}, draw=none, from=3-3, to=1-1]
	\arrow[hook, from=1-3, to=3-3]
	\arrow[hook, from=1-3, to=1-5]
	\arrow[hook, from=3-3, to=3-5]
	\arrow[hook, from=3-5, to=2-7]
	\arrow[hook, from=1-5, to=3-5]
	\arrow[hook, from=1-5, to=2-7]
	\arrow["\lrcorner"{anchor=center, pos=0.125, rotate=180}, draw=none, from=3-5, to=1-3]
\end{tikzcd}\] where $\bU_{v_{t_\gamma}}$ is a compact neighbourhood of $v_{t_\gamma}$. The characterisation of compact persistence modules as tame and point-wise finite dimensional implies that if $\bU_{v_{t_\gamma}}$ is compact, then so is $\bU_{\gamma_{s_\gamma}}$. Since $\gamma_r \in \bU_{\gamma_{s_\gamma}}$ for all $s_\gamma \leq r < t_\gamma$, it follows that $\overline{\V}$ is locally compact. 
\end{proof}

\subsection{Cofibrant persistent complexes} We prove now that a persistent cochain complex is cofibrant only if it is locally compact, while while tameness suffices to prove compactness (\Cref{injectivity-compactness}). Local compactness is defined  for persistent cochain complexes as for persistence modules, in terms of the existence of compact neighbourhoods. 

\begin{definition}
    A persistent complex $\X \in p\Ch_\Q$ is \textit{locally compact} if for each $x \in \X$ there exists some compact sub-complex $\W_x \in (p\Ch_\Q)^\omega \subseteq \X$ such that $x \in \X$.
\end{definition}

This definition is equivalent to degree-wise local compactness at the level of persistence modules.

\begin{lemma} \label{local-compactness-complexes}
    A persistent cochain complex $\X$ is locally compact if and only if  $\X^k \in p\mathsf{Vec}_\Q$ is a locally compact persistence module for all $k \geq 0$.
\end{lemma}

\begin{proof}
    If $\X$ is locally compact, then the restriction of the compact neighbourhood $\W_x \in p\Ch_\Q$ to degree $k$ is a compact neighbourhood $\W^k_x \in p\mathsf{Vec}_\Q$. On the other hand, suppose that $\X^k$ is locally compact for all $k \geq 0$, and consider the interval decomposition $\bU_x = \bigoplus_{i=0}^n \bI_{[s_i,t_i)}$ of the compact neighbourhood $\bU_x$ of $x$. Pick generators $x'_{s_i}$ for each interval $\bI_{[s_i,t_i)}$, and let $\bU_{dx'_{s_i}}$ be a compact neighbourhood of each $dx'_{s_i}$. The persistence complex $\W_x$ defined by the pushout
    % https://q.uiver.app/#q=WzAsNSxbMCwwLCJcXGJpZ29wbHVzX3tpPTB9Xm4gXFxTcF57aysxfV97W3NfaSx0X2kpfSJdLFswLDIsIlxcYmlnb3BsdXNfe2k9MH1ebiBcXERpXntrKzF9X3tbc19pLHRfaSl9Il0sWzIsMCwiXFxiaWdvcGx1c197aT0wfV5uIFxcYlVfe2R4X3tzX2l9fSJdLFsyLDIsIlxcV194Il0sWzMsMSwiXFxYIl0sWzAsMSwiIiwwLHsic3R5bGUiOnsidGFpbCI6eyJuYW1lIjoiaG9vayIsInNpZGUiOiJ0b3AifX19XSxbMCwyLCIoZHhfe3NfaX0sMCkiXSxbMSwzXSxbMiw0XSxbMiwzLCIiLDIseyJzdHlsZSI6eyJ0YWlsIjp7Im5hbWUiOiJob29rIiwic2lkZSI6InRvcCJ9fX1dLFszLDRdLFszLDAsIiIsMix7InN0eWxlIjp7Im5hbWUiOiJjb3JuZXIifX1dXQ==
\[\begin{tikzcd}[sep=small]
	{\bigoplus_{i=0}^n \Sp^{k+1}_{[s_i,t_i)}} && {\bigcup_{i=0}^n \bU_{dx'_{s_i}}} \\
	&&& \X, \\
	{\bigoplus_{i=0}^n \Di^{k+1}_{[s_i,t_i)}} && {\W_x}
	\arrow[hook, from=1-1, to=3-1]
	\arrow["{(dx_{s_i},0)}", from=1-1, to=1-3]
	\arrow[from=3-1, to=3-3]
	\arrow[hook, from=1-3, to=2-4]
	\arrow[hook, from=1-3, to=3-3]
	\arrow[hook,from=3-3, to=2-4]
	\arrow["\lrcorner"{anchor=center, pos=0.125, rotate=180}, draw=none, from=3-3, to=1-1]
\end{tikzcd}\] which is a compact neighbourhood of $x$ in $p\Ch_\Q$, since compactness is preserved under finite unions.
\end{proof}

\begin{comment}
\begin{lemma}
    If $\Y \in (p)^\omega$ and $0 \to \X \to \Y \to \bZ \to 0$ is an exact sequence then $\X,\bZ \in (p)^\omega$.
\end{lemma}

\begin{proof}
    We know that compact is equivalent to degree-wise compact, so it suffices to show that the condition holds on the induced exact sequence
    $$0 \to \X^k \to \Y^k \to \Z^k \to 0$$ of $p$-modules for each $k$. Compactness of $p$-modules is characterised by tame and point-wise finite dimensional. 
\end{proof}
\end{comment}

\begin{lemma} \label{pushout-local-compactness-complexes}
    If $\X \in p\Ch_\Q$ is locally compact in degrees $k > 0$, then the pushout
    % https://q.uiver.app/#q=WzAsNCxbMCwwLCJcXGJpZ29wbHVzX3tcXGdhbW1hIFxcaW4gXFxHYW1tYX0gXFxTcF57a19cXGdhbW1hKzF9X3tbc19cXGdhbW1hLHRfXFxnYW1tYSl9Il0sWzAsMiwiXFxiaWdvcGx1c197XFxnYW1tYSBcXGluIFxcR2FtbWF9IFxcRGlee2tfXFxnYW1tYSsxfV97c19cXGdhbW1hfSJdLFsyLDAsIlxcWCJdLFsyLDIsIlxcWSJdLFswLDFdLFswLDJdLFsxLDNdLFsyLDNdLFszLDAsIiIsMSx7InN0eWxlIjp7Im5hbWUiOiJjb3JuZXIifX1dXQ==
\[\begin{tikzcd}[sep=small]
	{\bigoplus_{\gamma \in \Gamma} \Sp^{k_\gamma+1}_{[s_\gamma,t_\gamma)}} && \X \\
	\\
	{\bigoplus_{\gamma \in \Gamma} \Di^{k_\gamma+1}_{s_\gamma}} && \Y
	\arrow[from=1-1, to=3-1]
	\arrow[from=1-1, to=1-3]
	\arrow[from=3-1, to=3-3]
	\arrow[from=1-3, to=3-3]
	\arrow["\lrcorner"{anchor=center, pos=0.125, rotate=180}, draw=none, from=3-3, to=1-1]
\end{tikzcd}\] is locally compact in degrees $k > 0$.
\end{lemma}

\begin{proof}
    By \Cref{local-compactness-complexes} it suffices to check local compactness degree-wise. With our assumptions, this reduces to checking that local compactness is preserved under the following pushout
    % https://q.uiver.app/#q=WzAsNCxbMCwwLCJcXGJpZ29wbHVzX3tcXGdhbW1hIFxcaW4gXFxHYW1tYV5rfSBcXGJJX3tbdF9cXGdhbW1hLFxcaW5mdHkpfV57a19cXGdhbW1hfSJdLFswLDIsIlxcYmlnb3BsdXNfe1xcZ2FtbWEgXFxpbiBcXEdhbW1hXmt9IFxcYklfe1tzX1xcZ2FtbWEsXFxpbmZ0eSl9XmsiXSxbMiwwLCJcXFgiXSxbMiwyLCJcXFkiXSxbMCwyXSxbMSwzXSxbMiwzLCIiLDEseyJzdHlsZSI6eyJ0YWlsIjp7Im5hbWUiOiJob29rIiwic2lkZSI6InRvcCJ9fX1dLFswLDEsIiIsMSx7InN0eWxlIjp7InRhaWwiOnsibmFtZSI6Imhvb2siLCJzaWRlIjoidG9wIn19fV0sWzMsMCwiIiwwLHsic3R5bGUiOnsibmFtZSI6ImNvcm5lciJ9fV1d
\[\begin{tikzcd}[sep=small]
	{\bigoplus_{\gamma \in \Gamma^k} \bI_{[t_\gamma,\infty)}^{k_\gamma}} && \X \\
	\\
	{\bigoplus_{\gamma \in \Gamma^k} \bI_{[s_\gamma,\infty)}^{k_\gamma}} && \Y,
	\arrow[from=1-1, to=1-3]
	\arrow[from=3-1, to=3-3]
	\arrow[hook, from=1-3, to=3-3]
	\arrow[hook, from=1-1, to=3-1]
	\arrow["\lrcorner"{anchor=center, pos=0.125, rotate=180}, draw=none, from=3-3, to=1-1]
\end{tikzcd}\]
    which follows from \Cref{pmod-preserving-lcocal-compactness}. 
\end{proof}

\begin{lemma} \label{local-compact-implies-Icell}
    If $\X$ is an $\bI$-cell complex, then $\X$ is locally compact in degrees $k > 0$.
\end{lemma}

\begin{proof}
    Let $\X$ be a sequential colimit, $\X = \colim_i \X_i$, built from pushouts
    % https://q.uiver.app/#q=WzAsNCxbMCwwLCJcXGJpZ29wbHVzX3tcXGdhbW1hIFxcaW4gXFxHYW1tYV9pfSBcXFNwXntrX1xcZ2FtbWErMX1fe1tzX1xcZ2FtbWEsdF9cXGdhbW1hKX0iXSxbMCwyLCJcXGJpZ29wbHVzX3tcXGdhbW1hIFxcaW4gXFxHYW1tYV9pfSBcXERpXntrX1xcZ2FtbWErMX1fe3NfXFxnYW1tYX0iXSxbMiwwLCJcXFhfaSJdLFsyLDIsIlxcWF97aSsxfSJdLFswLDFdLFswLDJdLFsxLDNdLFsyLDNdLFszLDAsIiIsMSx7InN0eWxlIjp7Im5hbWUiOiJjb3JuZXIifX1dXQ==
\[\begin{tikzcd}[sep=small]
	{\bigoplus_{\gamma \in \Gamma_i} \Sp^{k_\gamma+1}_{[s_\gamma,t_\gamma)}} && {\X_i} \\
	\\
	{\bigoplus_{\gamma \in \Gamma_i} \Di^{k_\gamma+1}_{s_\gamma}} && {\X_{i+1}}
	\arrow[hook,from=1-1, to=3-1]
	\arrow[from=1-1, to=1-3]
	\arrow[from=3-1, to=3-3]
	\arrow[hook,from=1-3, to=3-3]
	\arrow["\lrcorner"{anchor=center, pos=0.125, rotate=180}, draw=none, from=3-3, to=1-1]
\end{tikzcd}\] with $k_\gamma \geq 1$ for all $i \in \N$, i.e.,  $\X= \bigcup_{i=0}^\infty \X_i$. By \Cref{pushout-local-compactness-complexes}, each $\X_i$ is locally compact by induction. Since each element $x \in \X$ must live in $\X_i$ for some $i$, it admits a compact neighbourhood $$\W_x \hookrightarrow \X_i \hookrightarrow \X,$$ showing that $\X$ is locally compact.
\end{proof}

\begin{lemma} \label{right-closed}
    If $\Y \in p\Ch_\Q$ and $\Y^k$ contains a right-closed $y(s) \in \Y^k(s)$, then $\Y$ is not cofibrant.
\end{lemma}

\begin{proof}
    An object $\Y$ is cofibrant if and only if it is a retract $i : \Y \hookrightarrow \X$ of an $\bI$-cell complex $\X$, which is locally compact by \Cref{local-compact-implies-Icell}. The image $i(y_s) \neq 0$ of $y_s$ is a right-closed point of $\X^k$, which contradicts the local compactness of $\X$ by \Cref{right-closed}.
\end{proof}

\begin{example}[Non-equivalence with the injective model structure] \label{non-injective-model-example}
    \rm{For $s \leq t$, the closed interval $\bI_{[s,t]} \otimes S^k$ for $k > 0$ is \textit{not} cofibrant, since it has a right-closed point at $t$. Since all objects are cofibrant in the usual model structure on $\Ch_\Q$, all objects must be cofibrant in the injective model structure on $p\Ch_\Q$, demonstrating non-equivalence with our model structure. }
\end{example}

\begin{remark} 
    Since there are no non-trivial maps $\bI_{[p,q)} \hookrightarrow \bI_{[s,t]}$ for $s\leq p<q\leq t$, the closed interval $\bI_{[s,t]}$ contains no non-trivial compact sub-modules and thus fails (spectacularly) to be locally compact. 
\end{remark}

Given that cofibrant implies locally compact (above degree 0), a natural question is whether the converse is true. Over the tame objects, which are a sub-class of locally compact persistence complexes, the converse holds. 

\begin{lemma} \label{injectivity-compactness}
    If $i : \X \to \Y$ is an injective map of tame $\X,\Y \in (p\Ch_\Q)^{tame}$, then $i$ is a cofibration in the interval sphere model structure. 
\end{lemma}

\begin{proof}
    Tame persistence modules form an abelian category \cite{giunti_2021}, so the space of cocycles $Z\X = \Ker(d : \X \to \X[1])$ is tame. The map reduces to an injective map $i : Z\X \to Z\Y$ on cocycles. Since $Z\X$ and $Z\Y$ are tame, the quotient is tame and thus has a graded interval decompostion $Z\Y/i(Z\X) = \oplus_\gamma \bI_{[s_\gamma,t_\gamma)}^{k_\gamma}$ at the level of graded persistence modules. For each interval pick generators $y_{s_\gamma} \in Z\Y/i(Z\X)(s_\gamma)$ and note that $y_{t_\gamma}$ must be in $i(Z\X)$ by the interval decomposition. We add missing cocycles into $i(\X)$ via the pushout
    % https://q.uiver.app/#q=WzAsNCxbMCwwLCJcXGJpZ29wbHVzX3tcXGdhbW1hIFxcaW4gXFxHYW1tYX0gXFxTcF97W3NfXFxnYW1tYSx0X1xcZ2FtbWEpfV57a19cXGdhbW1hKzF9Il0sWzAsMiwiXFxiaWdvcGx1c197XFxnYW1tYSBcXGluIFxcR2FtbWF9IFxcRGlfe3NfXFxnYW1tYX1ee2tfXFxnYW1tYSsxfSJdLFsyLDAsImkoXFxYKSJdLFsyLDIsImkoXFxYKVxcY3VwIFpcXFkiXSxbMCwxLCIiLDAseyJzdHlsZSI6eyJ0YWlsIjp7Im5hbWUiOiJob29rIiwic2lkZSI6InRvcCJ9fX1dLFswLDIsIigwLHlfe3RfXFxnYW1tYX0pIl0sWzEsMywieV97c19cXGdhbW1hfSJdLFsyLDMsIiIsMix7InN0eWxlIjp7InRhaWwiOnsibmFtZSI6Imhvb2siLCJzaWRlIjoidG9wIn19fV0sWzMsMCwiIiwwLHsic3R5bGUiOnsibmFtZSI6ImNvcm5lciJ9fV1d
\[\begin{tikzcd}[sep=small]
	{\bigoplus_{\gamma \in \Gamma} \Sp_{[s_\gamma,t_\gamma)}^{k_\gamma+1}} && {i(\X)} \\
	\\
	{\bigoplus_{\gamma \in \Gamma} \Di_{s_\gamma}^{k_\gamma+1}} && {i(\X)\cup Z\Y}
	\arrow[hook, from=1-1, to=3-1]
	\arrow["{(0,y_{t_\gamma})}", from=1-1, to=1-3]
	\arrow["{y_{s_\gamma}}", from=3-1, to=3-3]
	\arrow[hook, from=1-3, to=3-3]
	\arrow["\lrcorner"{anchor=center, pos=0.125, rotate=180}, draw=none, from=3-3, to=1-1]
\end{tikzcd}\] which is clearly still tame. Since quotients of tame objects by tame subobjects are tame, assign an interval decomposition 
$$\Y/\big(i(\X) \cup Z\Y\big) = \bigoplus_{\delta \in \Delta} \bI_{[s_\delta,t_\delta)}^{k_\delta}$$ at the level of graded persistence modules. Assigning generators $y_{s_\delta} \in \Y$ to each interval, we have that $dy_{s_\delta} \in Z \Y$ by the cochain condition and that $y_{t_{\delta}} \in i(\X) \cup Z\Y$ by the interval decomposition. We add in the rest of $\Y$ via the pushout
% % https://q.uiver.app/#q=WzAsNCxbMCwwLCJcXGJpZ29wbHVzX3tcXGRlbHRhIFxcaW4gXFxEZWx0YX0gXFxTcF97W3NfXFxkZWx0YSx0X1xcZGVsdGEpfV57a19cXGRlbHRhKzF9Il0sWzAsMiwiXFxiaWdvcGx1c197XFxkZWx0YSBcXGluIFxcRGVsdGF9IFxcRGlfe3NfXFxkZWx0YX1ee2tfXFxkZWx0YSsxfSJdLFsyLDAsImkoXFxYKVxcY3VwIFpcXFkiXSxbMiwyLCJcXFkiXSxbMCwxLCIiLDAseyJzdHlsZSI6eyJ0YWlsIjp7Im5hbWUiOiJtb25vIn19fV0sWzAsMiwiKGR5X3tzX1xcZGVsdGF9LHlfe3RfXFxkZWx0YX0pIl0sWzEsMywieV97c19cXGRlbHRhfSJdLFsyLDMsIiIsMix7InN0eWxlIjp7InRhaWwiOnsibmFtZSI6Im1vbm8ifX19XSxbMywwLCIiLDAseyJzdHlsZSI6eyJuYW1lIjoiY29ybmVyIn19XV0=
\[\begin{tikzcd}[sep=small]
	{\bigoplus_{\delta \in \Delta} \Sp_{[s_\delta,t_\delta)}^{k_\delta+1}} && {i(\X)\cup Z\Y} \\
	\\
	{\bigoplus_{\delta \in \Delta} \Di_{s_\delta}^{k_\delta+1}} && \Y
	\arrow[tail, from=1-1, to=3-1]
	\arrow["{(dy_{s_\delta},y_{t_\delta})}", from=1-1, to=1-3]
	\arrow["{y_{s_\delta}}", from=3-1, to=3-3]
	\arrow[tail, from=1-3, to=3-3]
	\arrow["\lrcorner"{anchor=center, pos=0.125, rotate=180}, draw=none, from=3-3, to=1-1]
\end{tikzcd}\] This realises $i : \X \to \Y$ as a cofibration via the composition $\X \xrightarrow{\cong} i(\X) \rightarrowtail i(\X) \cup Z\Y \rightarrowtail \Y.$
\end{proof}

\begin{corollary}
        If $\X \in p\Ch_\Q$ is tame then $\X$ is cofibrant in the interval sphere model structure. 
\end{corollary}

\begin{example} \rm{The open interval $\bI_{(s,t)} \otimes S^k \in p\Ch_\Q$ for $k \geq 1$ is an example of a cofibrant object which is locally compact but not tame. Such an object can be constructed via a sequence of attachments along $\Sp^{k+1}_{[s + \epsilon_i, t)} \hookrightarrow \Di^{k+1}_{s + \epsilon_i}$ for any descreasing sequence $(\epsilon_i \mid i \in \N, \epsilon_i \leq t-s)$ that converges to $0$. We conjecture that locally compact fully characterises the connected cofibrant objects, but leave this open for future work.}
\end{example}

\begin{lemma} \label{injectivity-compactness}
    If $i : \X \to \Y$ is an injective map of tame, connected objects $\X,\Y \in p\Ch_\Q$, then $i$ is a cofibration in the interval-sphere model structure. 
\end{lemma}

\begin{proof}
    The map $i$ induces an injective map $i : Z\X \to Z\Y$ on cocycles. Since $\X$ and $\Y$ are tame, the quotient admits an interval decompostion $Z\Y/i(Z\X) = \oplus_\gamma \bI_{[s_\gamma,t_\gamma)}^{k_\gamma}$ at the level of graded persistence modules. For each interval pick a generator $y_{s_\gamma} \in Z\Y/i(Z\X)(s_\gamma)$, and note that $y_{t_\gamma}$ must be in $i(Z\X)$ by the interval decomposition. To add to $i(\X)$ the cocycles of $\Y$ not in $i(z\X)$, form the pushout
    % https://q.uiver.app/#q=WzAsNCxbMCwwLCJcXGJpZ29wbHVzX3tcXGdhbW1hIFxcaW4gXFxHYW1tYX0gXFxTcF97W3NfXFxnYW1tYSx0X1xcZ2FtbWEpfV57a19cXGdhbW1hKzF9Il0sWzAsMiwiXFxiaWdvcGx1c197XFxnYW1tYSBcXGluIFxcR2FtbWF9IFxcRGlfe3NfXFxnYW1tYX1ee2tfXFxnYW1tYSsxfSJdLFsyLDAsImkoXFxYKSJdLFsyLDIsImkoXFxYKVxcdW5kZXJzZXR7aShaXFxYKX17XFxvcGx1c30gWlxcWSJdLFswLDEsIiIsMCx7InN0eWxlIjp7InRhaWwiOnsibmFtZSI6Imhvb2siLCJzaWRlIjoidG9wIn19fV0sWzAsMiwiKDAseV97dF9cXGdhbW1hfSkiXSxbMSwzLCJ5X3tzX1xcZ2FtbWF9Il0sWzIsMywiIiwyLHsic3R5bGUiOnsidGFpbCI6eyJuYW1lIjoiaG9vayIsInNpZGUiOiJ0b3AifX19XSxbMywwLCIiLDAseyJzdHlsZSI6eyJuYW1lIjoiY29ybmVyIn19XV0=
\[\begin{tikzcd}[sep=small]
	{\bigoplus_{\gamma \in \Gamma} \Sp_{[s_\gamma,t_\gamma)}^{k_\gamma+1}} && {i(\X)} \\
	\\
	{\bigoplus_{\gamma \in \Gamma} \Di_{s_\gamma}^{k_\gamma+1}} && {i(\X)\underset{i(Z\X)}{\oplus} Z\Y,}
	\arrow[tail, from=1-1, to=3-1]
	\arrow["{(0,y_{t_\gamma})}", from=1-1, to=1-3]
	\arrow["{y_{s_\gamma}}", from=3-1, to=3-3]
	\arrow[tail, from=1-3, to=3-3]
	\arrow["\lrcorner"{anchor=center, pos=0.125, rotate=180}, draw=none, from=3-3, to=1-1]
\end{tikzcd}\] which is clearly still tame. Consider the interval decomposition 
$$\Y/\big(i(\X)\underset{i(Z\X)}{\oplus} Z\Y\big) = \bigoplus_{\delta \in \Delta} \bI_{[s_\delta,t_\delta)}^{k_\delta}$$ at the level of graded persistence modules. Given generators $y_{s_\delta} \in \Y$ of each interval, it follows that $dy_{s_\delta} \in Z \Y$ and that $y_{t_{\delta}} \in i(\X) \oplus_{i(Z\X)} Z\Y$ by the interval decomposition. Form now the pushout
% https://q.uiver.app/#q=WzAsNCxbMCwwLCJcXGJpZ29wbHVzX3tcXGRlbHRhIFxcaW4gXFxEZWx0YX0gXFxTcF97W3NfXFxkZWx0YSx0X1xcZGVsdGEpfV57a19cXGRlbHRhKzF9Il0sWzAsMiwiXFxiaWdvcGx1c197XFxkZWx0YSBcXGluIFxcRGVsdGF9IFxcRGlfe3NfXFxkZWx0YX1ee2tfXFxkZWx0YSsxfSJdLFsyLDAsImkoXFxYKVxcdW5kZXJzZXR7aShaXFxYKX17XFxvcGx1c30gWlxcWSJdLFsyLDIsIlxcWSJdLFswLDEsIiIsMCx7InN0eWxlIjp7InRhaWwiOnsibmFtZSI6Im1vbm8ifX19XSxbMCwyLCIoZHlfe3NfXFxkZWx0YX0seV97dF9cXGRlbHRhfSkiXSxbMSwzLCJ5X3tzX1xcZGVsdGF9Il0sWzIsMywiIiwyLHsic3R5bGUiOnsidGFpbCI6eyJuYW1lIjoibW9ubyJ9fX1dLFszLDAsIiIsMCx7InN0eWxlIjp7Im5hbWUiOiJjb3JuZXIifX1dXQ==
\[\begin{tikzcd}[sep=small]
	{\bigoplus_{\delta \in \Delta} \Sp_{[s_\delta,t_\delta)}^{k_\delta+1}} && {i(\X)\underset{i(Z\X)}{\oplus} Z\Y} \\
	\\
	{\bigoplus_{\delta \in \Delta} \Di_{s_\delta}^{k_\delta+1}} && \Y
	\arrow[tail, from=1-1, to=3-1]
	\arrow["{(dy_{s_\delta},y_{t_\delta})}", from=1-1, to=1-3]
	\arrow["{y_{s_\delta}}", from=3-1, to=3-3]
	\arrow[tail, from=1-3, to=3-3]
	\arrow["\lrcorner"{anchor=center, pos=0.125, rotate=180}, draw=none, from=3-3, to=1-1]
\end{tikzcd}\] We have thus exhibited $i : \X \to \Y$ as a cofibration via the composite
$$\X \xrightarrow{\cong} i(\X) \rightarrowtail i(\X)\underset{i(Z\X)}{\oplus} Z\Y \rightarrowtail \Y.$$
\end{proof}

\begin{corollary}
        If $\X \in (p\Ch_\Q)^\tame$ is connected, then $\X$ is cofibrant in the interval-sphere model structure. 
\end{corollary}

\begin{proof}
    Apply the previous lemma to the injective map $0 \to \X$.
\end{proof}

\section{Transfer of model structures}

The interval-sphere model structure on $\Ch_\Q$ enables us to relate the homotopy theory of persistent CDGAs to that of homotopy theory of copersistent simplicial sets.

\subsection{Rational homotopy theory model structures} CDGAs and minimal models are the basic algebraic tools of rational homotopy theory. We briefly review the literature, focusing on the results that we generalize to the persistent setting.  

\subsubsection{CDGAs} Transfer along the free-forgetful adjunction
$$
\begin{tikzcd}
\Lambda : \Ch_\Q \ar[r, shift left] & \CDGA_\Q : U \ar[l, shift left]
\end{tikzcd}
$$
yields the combinatorial model structure constructed by \cite{bousfield1976pl} on the category $\CDGA_\Q$, for which the sets
$$\begin{aligned}
\Lambda I & = \{ \: \Lambda S^k \hookrightarrow \Lambda D^k \:,\: 0\to \Lambda S^0 \: | \: k\in \N \:\} \\
\Lambda J &= \{ \: \Lambda D^k \to 0 \: | \: k\in \N \:\}.
\end{aligned}$$
constitute generating cofibrations and trivial cofibrations.

\subsubsection{Simplicial sets} There is a model structure on $\sSet$ whose 
\begin{enumerate}
    \item cofibrations are monomorphisms,
    \item weak equivalences are $H^*(-;\Q)$-isomorphisms, i.e., maps that induce isomorphisms on singular cohomology with $\Q$-coefficients.
\end{enumerate}
This model structure is a left Bousfield localization of the Kan model structure on simplicial sets. We call it the \textit{$
\Q$-model structure} and call its fibrations and weak equivalences \textit{$\Q$-fibrations} and \textit{$\Q$-equivalences}, respectively.

\subsubsection{The key adjunction} Let $\Delta$ be the simplex category. There is a simplicial object $\Omega_\bullet : \Delta^{op} \to \CDGA_\Q$ that associates to each $[n]$ the CDGA 
$$\Omega_n = \Lambda(x_1,dx_1, \cdots, x_n, dx_n)$$ of polynomial differential forms with rational coefficients over the $n$-simplex with degrees $\lvert x_i \rvert = 0$ for all $i$. Kan extension of this simplicial object over the Yoneda embedding yields an adjunction
\[\begin{tikzcd}
            \langle - \rangle : \CDGA_\Q \arrow[r, shift left=1ex, ""{name=G}] & \sSet^\op : \apl \arrow[l, shift left=.5ex, ""{name=F}]
            \arrow[phantom, from=F, to=G, , "\scriptscriptstyle\boldsymbol{\bot}"]
\end{tikzcd}.\] The functor $\mathcal{A}_{PL}$ maps a simplicial set $X$ to that CDGA of polynomial differential forms. The \textit{spatial realization} functor $\langle - \rangle$ sends a CDGA $\mathcal{A}$ to the simplicial set of cosingular simplices $\langle \mathcal{A} \rangle = \text{Map}(\mathcal{A}, \Omega_\bullet)$.

\begin{theorem}[\cite{bousfield1976pl}, Section 8] \label{quillen_adjunction}
    The functors $\langle - \rangle$ and $\apl$ form a  Quillen adjunction where $\CDGA_\Q$ is endowed with the projective model structure and $\sSet^{op}$ with the (opposite of) the $\Q$-model structure. This adjunction induces an equivalence between the homotopy categories of the full subcategories of simply connected CDGAs of finite type $\CDGA_\Q^{\geq 2, ft} \subseteq \CDGA_\Q$ and that of simply connected spaces of finite type $\sSet_{\geq 2}^{ ft} \subseteq \sSet$. 
\end{theorem}

\subsection{Transfer of model structures} We begin with a transfer of the interval-sphere model structure to persistent CDGAs, followed by its comparison with a certain projective model structure on copersistent simplicial sets.

\subsubsection{Transferred model structure on $p\CDGA$} Post-composition of persistent objects with the free-forgetful  adjunction yields an adjunction
$$
\begin{tikzcd}
    \Lambda:p\Ch_\Q \arrow[r, shift left=1ex, "{}"{name=G}] & p\CDGA_\Q : U \arrow[l, shift left=.5ex, "{}"{name=F}]
    \arrow[phantom, from=G, to=F, "\scriptscriptstyle\boldsymbol{\bot}"]
\end{tikzcd}
$$ at the level of persistent objects.

The category $p\CDGA_\Q$ is locally finitely presentable, hence the small object argument can be applied to any set of maps. Its compact objects are the persistent CDGAs $\bA$ such that $\bA(t)$ is a $\Q$-algebra of finite presentation for every $t\in \R_+$. We denote by $U^{-1}\W$ the class of maps $f:\bA \to \bB$ in $p\CDGA_\Q$ such that $Uf \in \W$. In other words, $f(s)$ is a quasi-isomorphism for every $s\in \R_+$. We also write $$
\begin{aligned}
    \Lambda \I &=\big\{\: \Lambda \Sp^k_{[s,t)} \to \Lambda \Di^k_s \:|\: 0\leq s\leq t \:\big\} \\
    \Lambda \J &=\big\{\: \Lambda \Di^k_{t} \to \Lambda \Di^k_s \:|\: 0\leq s\leq t \:\big\}
\end{aligned}
$$
for the images of $\I,\J \subseteq p\Ch_\Q$ under the functor $\Lambda$.

\begin{theorem}\label{modelstructureCDGA}
    There is a combinatorial model structure on $p\CDGA_\Q$ for which $U^{-1}\W$ is the class of weak equivalences and the sets $\Lambda \I, \Lambda\J$ are, respectively, those of the generating cofibrations and trivial cofibrations.
\end{theorem}
\begin{proof}
    It suffices to check that the conditions for the recognition theorem of D. Kan \cite[Theorem~11.3.2]{HH} are satisfied. For this it is enough to show that $$
    \Cell(\Lambda\J)\subseteq U^{-1}\W
    $$
    where $\Cell(\Lambda \J)\subseteq p\CDGA_\Q$ denotes the class of relative $\Lambda \J$-cell complexes. Since $\Lambda\J\subseteq U^{-1}\W$, it suffices to show that the class $U^{-1}\W$ is closed under cobase change and transfinite composition. The latter condition is clear since the forgetful functor $U$ preserves filtered colimits. We thus need to prove that the class $\Lambda\J$ is closed under cobase change. To this end consider a pushout square 
    $$
    \begin{tikzcd}
        {\Lambda\Di^k_t} \ar[r] \ar[d] & \bA \ar[d] \\ 
        {\Lambda\Di^k_s} \ar[r] & \bB \arrow[ul, phantom, "\ulcorner", very near start]
    \end{tikzcd}
    $$
    in $p\CDGA_\Q$. One sees immediately that $\bA(u)\to \bB(u)$ is an isomorphism when $u<s$ or $t\leq u$. When $s\leq u < t$, we find $\bB(u)\cong \bA(u)\otimes \Lambda(a_u, da_u)$ so that the inclusion $\bA(u)\subseteq \bB(u)$ is a quasi-isomorphism.
\end{proof}

\subsection{Persistent minimal models} Let $[n] = \{ 0,1, \ldots, n\}$ be the usual poset category. In \cite{hess2024celldecompositionspersistentminimal}, we defined a persistent minimal model $\mathbb{M} : [n] \to \CDGA_\Q$ of a (finitely-indexed) simply connected, persistent CDGA $\mathbb{A} : [n] \to \CDGA_\Q$ as a homotopy commutative diagram
%https://q.uiver.app/#q=WzAsMTAsWzIsMCwiXFxtTShyKSJdLFs0LDAsIlxcbU0ocisxKSJdLFsyLDIsIlxcbUEocikiXSxbNCwyLCJcXG1BKHIrMSkiXSxbMCwwLCJcXGxkb3RzIl0sWzYsMCwiXFxtTShyKzIpIl0sWzYsMiwiXFxtQShyKzIpIl0sWzAsMiwiXFxsZG90cyJdLFs4LDAsIlxcbGRvdHMiXSxbOCwyLCJcXGxkb3RzIl0sWzAsMV0sWzAsMiwibShyKSJdLFsyLDNdLFsxLDMsIm0ocisxKSJdLFs0LDBdLFsxLDVdLFsxLDIsIkgocikiLDAseyJzaG9ydGVuIjp7InNvdXJjZSI6MjAsInRhcmdldCI6MjB9LCJsZXZlbCI6Mn1dLFs1LDYsIm0ocisyKSJdLFszLDZdLFs1LDMsIkgocisxKSIsMCx7InNob3J0ZW4iOnsic291cmNlIjoyMCwidGFyZ2V0IjoyMH0sImxldmVsIjoyfV0sWzcsMl0sWzUsOF0sWzYsOV1d
\begin{equation} \label{tame-minimal-model} \begin{tikzcd}[sep=scriptsize]
	\ldots && {\bM(r)} && {\bM(r+1)} && {\bM(r+2)} && \ldots \\
	\\
	\ldots && {\bA(r)} && {\bA(r+1)} && {\bA(r+2)} && \ldots
	\arrow[from=1-1, to=1-3]
	\arrow[from=1-3, to=1-5]
	\arrow["{m(r)}", from=1-3, to=3-3]
	\arrow[from=1-5, to=1-7]
	\arrow["{H(r)}", shorten <=13pt, shorten >=13pt, Rightarrow, from=1-5, to=3-3]
	\arrow["{m(r+1)}", from=1-5, to=3-5]
	\arrow[from=1-7, to=1-9]
	\arrow["{H(r+1)}", shorten <=14pt, shorten >=14pt, Rightarrow, from=1-7, to=3-5]
	\arrow["{m(r+2)}", from=1-7, to=3-7]
	\arrow[from=3-1, to=3-3]
	\arrow[from=3-3, to=3-5]
	\arrow[from=3-5, to=3-7]
	\arrow[from=3-7, to=3-9]
\end{tikzcd}\end{equation} in the category $\CDGA_\Q$ where the vertical maps are Sullivan minimal models \cite{sullivan}. We denote such a homotopy-commutative diagram by $m : \bM \xrightarrow{k}_2 \bA$ for $\bA$ for shorthand. We first re-state the following motivating theorem from \cite{hess2024celldecompositionspersistentminimal} for the use of interval spheres. 

\begin{theorem}[4.14, \cite{hess2024celldecompositionspersistentminimal}]\label{structure-theorem}
    Every simply-connected persistent CDGA $\bA : [n] \to \CDGA$ over finite index $[n]$ admits a persistent minimal model $\bM$ whose skeletal filtration is constructed as a sequence of degree $k$ persistent Hirsch extensions
    % https://q.uiver.app/#q=WzAsNSxbMCwwLCJcXExhbWJkYSBcXGJTXntrKzF9KEhea1xcQ19tKSJdLFswLDIsIlxcTGFtYmRhIFxcYkRee2srMX0oSF5rXFxDX20pIl0sWzIsMCwiXFxiTV97ay0xfSJdLFsyLDIsIlxcYk1fayJdLFs0LDAsIlxcYkEiXSxbMCwxLCIiLDAseyJzdHlsZSI6eyJ0YWlsIjp7Im5hbWUiOiJob29rIiwic2lkZSI6InRvcCJ9fX1dLFswLDJdLFsxLDNdLFsyLDMsIiIsMix7InN0eWxlIjp7InRhaWwiOnsibmFtZSI6Imhvb2siLCJzaWRlIjoidG9wIn19fV0sWzIsNCwibSJdLFszLDQsIm1fayIsMix7ImN1cnZlIjozfV0sWzIsNCwiMiIsMix7ImxhYmVsX3Bvc2l0aW9uIjo4MH1dLFszLDQsIjIiLDIseyJsYWJlbF9wb3NpdGlvbiI6OTAsImN1cnZlIjozfV0sWzMsMCwiIiwwLHsic3R5bGUiOnsibmFtZSI6ImNvcm5lciJ9fV1d
\[\begin{tikzcd}
	{\Lambda \Sp^{k+1}(H^k\C_m)} && {\bM_{k-1}} && \bA \\
	\\
	{\Lambda \Di^{k+1}(H^k\C_m)} && {\bM_k}
	\arrow[from=1-1, to=1-3]
	\arrow[hook, from=1-1, to=3-1]
	\arrow["m", from=1-3, to=1-5]
	\arrow["2"'{pos=0.8}, from=1-3, to=1-5]
	\arrow[hook, from=1-3, to=3-3]
	\arrow[from=3-1, to=3-3]
	\arrow["\lrcorner"{anchor=center, pos=0.125, rotate=180}, draw=none, from=3-3, to=1-1]
	\arrow["{m_k}"', from=3-3, to=1-5]
	\arrow["2"'{pos=0.9}, from=3-3, to=1-5]
\end{tikzcd}\]
\end{theorem}

The results holds equally well for any simply connected, tame CDGA $\widetilde{\mathbb{A}} : \R_+ \to \CDGA$ since, by definition, it is a Left Kan extension
\[
\begin{tikzcd}
        {\{t_0  < \ldots < t_n\}} \ar[dr, "\mathbb{A}"]\ar[d, "t"'] & \\
        {\R_+} \ar[r, "{\widetilde{\mathbb{A}}=\Lan_t \mathbb{A}}"'] & \CDGA_\Q
    \end{tikzcd}
\] and $\{t_0  < \ldots < t_n\} \cong [n]$ as poset categories. The following corollary -- a direct consequence of \ref{modelstructureCDGA} and \ref{structure-theorem} -- is a generalization of the classical cell complex definition (Cf. \cite{hess2024celldecompositionspersistentminimal}) of minimal models.  

\begin{corollary}\label{persistent-min-moodel}
    Persistent minimal models are $\bI$-cell complexes in the transferred interval-sphere model structure of \Cref{modelstructureCDGA}. 
\end{corollary}

\subsubsection{Model structure on $p^*\sSet$} We now establish a persistent variant of the classical Quillen pair, setting the ground for the study of \textit{persistent rational homotopy theory}. Denote by
$$
p^*\sSet = \Fun(\R_+^\op, \sSet)\cong\Fun(\R_-, \sSet)
$$ 
the category of \textit{copersistent} spaces. As above, the functors $\langle - \rangle$ and $\apl$ can be extended to persistent CDGAs and copersistent spaces, respectively. This yields the following contravariant adjunction:
$$
\begin{tikzcd}
    \langle - \rangle :p\CDGA_\Q \arrow[r, shift left=1ex, "{}"{name=G}] & (p^*\sSet)^\op : \apl .\arrow[l, shift left=.5ex, "{}"{name=F}]
    \arrow[phantom, from=G, to=F, "\scriptscriptstyle\boldsymbol{\bot}"]
\end{tikzcd}
$$

Given $\sSet$ is cofibrantly generated, the category $p^*\sSet = \Fun(\R_-, \sSet)$ of copersistent spaces admits the projective $\Q$-model structure \cite[11.6.1]{HH}, which is described as follows. A map $f:\X \to \Y$ in $p^* \sSet$ is
\begin{enumerate}
    \item a fibration if $f(s):\X(s)\to \Y(s)$ is a $\Q$-fibration for every $s\leq 0$,
    \item a weak equivalence if $f(s):\X(s)\to \Y(s)$ is a $\Q$-equivalence for every $s\leq 0$.
\end{enumerate}

\subsection{The Quillen pair} Finally, we show that our new model structure forms a Quillen pair when compared with the above projective model structure on copersistent simplicial sets. 

\begin{proposition} \label{persistent-quillen-pair}
    The contravariant adjunction
    $$
    \begin{tikzcd}
        \langle - \rangle :p\CDGA_\Q \arrow[r, shift left=1ex, "{}"{name=G}] & (p^*\sSet)^\op : \apl \arrow[l, shift left=.5ex, "{}"{name=F}]
        \arrow[phantom, from=G, to=F, "\scriptscriptstyle\boldsymbol{\bot}"]
    \end{tikzcd}
    $$
    is a Quillen pair when $(p^*\sSet)^\op$ is endowed with the opposite of the projective $\Q$-model structure, and $p\CDGA_\Q$ is given the interval-sphere model structure of \Cref{modelstructureCDGA}.
\end{proposition}
\begin{proof}
    It suffices to check that $\langle - \rangle$ sends the generating (trivial) cofibrations in $p\CDGA_\Q$ to (trivial) $\Q$-fibrations in $p^*\sSet$. Given a map $f \in \Lambda \I$ of the form $\Lambda\Sp^k_{[s,t)}\to \Lambda\Di^k_s$, its component $\langle f \rangle (u)=\langle f(u) \rangle$ is the map induced on spatial realizations by the Hirsch extension 
    $$
    \Lambda S^k \subseteq \Lambda D^k 
    $$
    when $s\leq u < t$ and an isomorphism otherwise. These maps are all $\Q$-fibrations as a consequence of \Cref{quillen_adjunction}. Similarly, the components of any $f\in \Lambda \J$ are Hirsch extensions and quasi-isomorphisms, hence map to trivial $\Q$-fibrations under spatial realization. 
\end{proof}

\bibliography{misc/bibliography}
\bibliographystyle{alpha} 
\end{document}